\documentclass[12pt]{article}
\usepackage{amsmath}
\usepackage{amssymb}
\usepackage{amsfonts}
\usepackage{amsmath,amsthm,amssymb}
\allowdisplaybreaks

\newcommand{\ds}{\displaystyle}

\theoremstyle{plain}

\def \[{\begin{equation}}
\def \]{\end{equation}}

\newtheorem{thm}{Theorem}[section]

\newtheorem{lem}[thm]{Lemma}

\newtheorem{prop}[thm]{Proposition}

\newtheorem{rem}[thm]{Remark}

\newif \ifLastSection \LastSectionfalse

\numberwithin{equation}{section}

\newcommand{\R}{{\mathbb R}}

\begin{document}
\title{   Existence of positive ground state solutions for the nonlinear Kirchhoff type equations in $\R^{3}$
\thanks {a: Partially  supported  by  NSFC No: 11071095 and  Hubei Key Laboratory of Mathematical
Sciences.  }
\thanks{b:Corresponding author: Gongbao Li, School of Mathematics and
Statistics, Central China  Normal University,  Wuhan, 430079, China.
E-mail address: ligb@mail.ccnu.edu.cn}}
\author{Gongbao Li ,\,\,\,\,\,Hongyu Ye \\ \small {School of Mathematics and Statistics,  Central China Normal
University,}
\\
\small{ Wuhan 430079, P. R. China}}
\date{}
\maketitle

\begin{abstract}
In this paper, we study the following nonlinear problem of Kirchhoff
type with pure power nonlinearities:
\begin{equation}\label{0.1}
\left\{%
\begin{array}{ll}
    -\left(a+b\ds\int_{\R^3}|D u|^2\right)\Delta u+V(x)u=|u|^{p-1}u, & \hbox{$x\in \R^3$}, \\
    u\in H^1(\R^3),~~~~u>0, & \hbox{$x\in \R^3$},\\
\end{array}%
\right.\end{equation}
 where $a,$ $b>0$ are constants, $2<p<5$ and
$V:\R^3\rightarrow\R$. Under certain assumptions on $V$, we prove
that \eqref{0.1} has a positive ground state solution by using a
monotonicity trick and a new version of global compactness lemma.

Our main results can be viewed as a partial extension of a recent
result of He and Zou in \cite{hz} concerning the existence of
positive solutions to the nonlinear Kirchhoff problem
$$\left\{%
\begin{array}{ll}
    -\left(\varepsilon^2a+\varepsilon b\ds\int_{\R^3}|D u|^2\right)\Delta u+V(x)u=f(u), & \hbox{$x\in \R^3$}, \\
    u\in H^1(\R^3),~~~~u>0, & \hbox{$x\in \R^3$},\\
\end{array}%
\right.$$where $\varepsilon>0$ is a parameter, $V(x)$ is a positive
continuous potential and $f(u)\thicksim |u|^{p-1}u$ with $3<p<5$ and
satisfies the Ambrosetti-Rabinowitz type condition. Our main results
extend also the arguments used in \cite{ap,zz}, which deal with
Schr\"{o}dinger-Poisson system with pure power nonlinearities, to
the Kirchhoff type problem.
\end{abstract}

\section{ Introduction and main result}
In this paper, we consider the existence of positive ground state
solutions to the following Kirchhoff type problem with pure power
nonlinearities:
\begin{equation}\label{1.1}
\left\{%
\begin{array}{ll}
    -\left(a+b\ds\int_{\R^3}|D u|^2\right)\Delta u+V(x)u=|u|^{p-1}u, & \hbox{$x\in \R^3$}, \\
    u\in H^1(\R^3),~~~~u>0, & \hbox{$x\in \R^3$},\\
\end{array}%
\right.\end{equation} where $a,~b>0$ are constants and $2<p<5$. We
assume that $V(x)$ verifies the following hypotheses:

$(V_1)$~~~$V(x)\in C(\R^3,\R)$ is weakly differentiable and
satisfies $(DV(x),x)\in L^\infty(\R^3)\cup L^{\frac{3}{2}}(\R^3)$
and
$$V(x)-(DV(x),x)\geq0~~a.e.~x\in\R^3,$$
where $(\cdot,\cdot)$ is the usual inner product in $\R^3$;

$(V_2)$~~~for almost every $x\in\R^3$,
$V(x)\leq\liminf\limits_{|y|\rightarrow+\infty}V(y)\triangleq
V_\infty<+\infty$ and the inequality is strict in a subset of
positive Lebesgue measure;

$(V_3)$~~~there exists a $\overline{C}>0$ such that
$$\overline{C}=\inf_{u\in H^1(\R^3)\backslash\{0\}}\frac{\int_{\R^3}|Du|^2+V(x)|u|^2}{\int_{\R^3}|u|^2}>0.$$

In recent years, the following elliptic problem
\begin{equation}\label{1.2}\left\{%
\begin{array}{ll}
    -\left(a+b\ds\int_{\R^N}|D u|^2\right)\Delta u+V(x)u=f(x,u), & \hbox{$x\in \R^N$}, \\
    u\in H^1(\R^N) & \hbox{$ $}\\
\end{array}%
\right.\end{equation} has been studied extensively by many
researchers, where $V:\R^N\rightarrow\R$, $f\in C(\R^N\times\R,\R)$,
$N=1,2,3$ and $a,b>0$ are constants. \eqref{1.2} is a nonlocal
problem as the appearance of the term  $\int_{\R^N}|D u|^2$ implies
that \eqref{1.2} is not a pointwise identity. This causes some
mathematical difficulties which make the study of \eqref{1.2}
particularly interesting. Problem \eqref{1.2} arises in an
interesting physical context. Indeed, if we set $V(x)=0$ and replace
$\R^N$ by a bounded domain $\Omega\subset\R^N$ in \eqref{1.2}, then
we get the following Kirchhoff Dirichlet problem
$$\left\{%
\begin{array}{ll}
    -\left(a+b\ds\int_{\Omega}|D u|^2\right)\Delta u=f(x,u), & \hbox{$x\in\Omega$}, \\
    u=0, & \hbox{$x\in \partial\Omega$},\\
\end{array}%
\right.$$ which is related to the stationary analogue of the
equation
\begin{equation}\label{1.3}
\rho\frac{\partial^2u}{\partial
t^2}-\left(\frac{P_0}{h}+\frac{E}{2L}\ds\int_0^L\left|\frac{\partial
u}{\partial x}\right|^2dx\right)\frac{\partial^2 u}{\partial x^2}=0
\end{equation}presented by Kirchhoff in \cite{k}. The readers can learn some early research of Kirchhoff equations from \cite{b,p}.
In \cite{l}, J. L. Lions introduced an abstract functional analysis
framework to the following equation
\begin{equation}\label{1.4}
u_{tt}-\left(a+b\ds\int_{\Omega}|D u|^2\right)\Delta
u=f(x,u).\end{equation}After that, \eqref{1.4} received much
attention, see \cite{ac,af,ap1,cd,ds} and the references therein.

Before we review some results about \eqref{1.2}, we give several
definitions.

Let $(X,\parallel\cdot\parallel)$ be a Banach space with its dual
space $(X^*,\parallel\cdot\parallel_*)$, $I\in C^1(X,\R)$ and
$c\in\R$. We say a sequence $\{x_n\}$ in $X$ a Palais-Smale sequence
at level $c$ ($(PS)_{c}$ sequence in short) if $I(x_{n})\rightarrow
c$ and $\parallel I'(x_{n})\parallel_*\rightarrow 0$ as
$n\rightarrow\infty.$ We say
that $I$ satisfies $(PS)_c$ condition if for any
$(PS)_c$ sequence $\{x_n\}$ in $X$, there exists a
subsequence $\{x_{n_k}\}$ such that $x_{n_k}\rightarrow x_0$ in $X$
for some $x_0\in X$.

Throughout the paper, we use standard notations. For simplicity, we
write $\int_{\Omega} h$ and $\int_{\partial\Omega} hdS$ to mean the
Lebesgue integral of $h(x)$ over a domain $\Omega\subset\R^3$ and
over its boundary $\partial\Omega$ respectively. $L^{p}\triangleq
L^{p}(\R^{3})~(1\leq p<+\infty)$ is the usual Lebesgue space with
the standard norm $|u|_{p}.$ We use "$\rightarrow$" and
"$\rightharpoonup$" to denote the strong and weak convergence in the
related function space respectively.
$B_r(x)\triangleq\{y\in\R^3|~|x-y|<r\}$. $C$ will denote a positive
constant unless specified.

There have been many works about the existence of nontrivial
solutions to \eqref{1.2} by using variational methods, see e.g.
\cite{hz,jw,lh,wt,w}. Clearly weak solutions to \eqref{1.2}
correspond to critical points of the energy functional
$$\Psi(u)=\frac{1}{2}\ds\int_{\R^N}(a|Du|^2+V(x)|u|^2)+\frac{b}{4}\left(\ds\int_{\R^N}|Du|^2\right)^2-\ds\int_{\R^N}F(x,u)$$
defined on $E\triangleq\{u\in
H^1(\R^N)|~\int_{\R^N}V(x)|u|^2<\infty\}$, where
$F(x,u)=\int_{0}^{u}f(x,s)ds$. A typical way to deal with
\eqref{1.2} is to use the mountain-pass theorem. For this purpose,
one usually assumes that $f(x,u)$ is subcritical, superlinear at the
origin and either 4-superlinear at infinity in the sense that
$$\lim\limits_{|u|\rightarrow+\infty}\frac{F(x,u)}{u^4}=+\infty~~\hbox{uniformly~in~}x\in\R^N$$
or satisfies the Ambrosetti-Rabinowitz type condition ($(AR)$ in
short):
$$(AR)~~~~~~~~~~~~\exists~\mu>4~~\hbox{such~that}~~~~0<\mu~F(x,u)\leq f(x,u)u~~~~\hbox{for~all~}u\neq0.~~~~~~~~~~~~~~~~~~~~$$
Under the above mentioned conditions, one easily sees that $\Psi$
possesses a mountain-pass geometry around $0\in H^1(\R^N)$ and by
the mountain-pass theorem, one can get a $(PS)$ sequence of $\Psi$.
Moreover, the $(PS)$ sequence is bounded if
$$(F)~~~~~~~~~~~~~~~~~~~~~~~~~~~~~~~~~~~~~~~~~~4F(x,u)\leq f(x,u)u~~\hbox{for~all}~u\in\R~~~~~~~~~~~~~~~~~~~~~~~~~~~~~~~~~~~~~~~~~~~~$$
holds. Therefore, one can show that $\Psi$ satisfies the $(PS)$
condition and \eqref{1.2} has at least one nontrivial solution
provided some further conditions on $f(x,u)$ and $V(x)$ are assumed
to guarantee the compactness of the $(PS)$ sequence.

In \cite{jw}, Jin and Wu proved that \eqref{1.2} has infinitely many
radial solutions by using a fountain theorem when $N=2,3$,
$V(x)\equiv1$ and $f(x,u)$ is subcritical, superlinear at the origin
and 4-superlinear at infinity and invariant with respect to $x\in
\R^N$ under the actions of group of orthogonal transformations,
together with some conditions which are weaker than $(AR)$.

In \cite{w}, Wu obtained the existence of nontrivial solutions to
\eqref{1.2} by proving that $(PS)$ condition holds when $f(x,u)$ is
4-superlinear at infinity and satisfies $(F)$ and other conditions,
the potential $V(x)\in C(\R^N,\R)$ satisfies
$$(V_4)~~~~~~\inf\limits_{x\in\R^N} V(x)\geq a_1>0~~\hbox{and~for~each~}M>0,~\hbox{meas}\{x\in\R^N:~V(x)\leq
M\}<+\infty~~~~~~$$to ensure the compactness of embeddings of
$E\triangleq\{u\in H^1(\R^3)|~\int_{\R^3}V(x)|u|^2<\infty\}\hookrightarrow L^q(\R^N)$, where $2\leq q<2^*=\frac{2N}{N-2}$,
$a_1$ is a constant and $\hbox{meas}(\cdot)$ denotes the Lebesgue
measure in $\R^N$.

In \cite{hz}, He and Zou studied \eqref{1.2} under the conditions:
$N=3$, a positive continuous potential $V(x)$ satisfies
$$(V_5)~~~~~~~~~~~~~~~~~~~~~~~~~~~V_\infty=\liminf\limits_{|x|\rightarrow\infty}V(x)>V_0=\inf\limits_{x\in\R^3}V(x)>0,~~~~~~~~~~~~~~~~~~~~~~~~~$$
$f(x,u)=f(u)\in C^1(\R_+,\R_+)$ satisfies $(AR)$,
$\lim\limits_{|u|\rightarrow0}\frac{f(u)}{|u|^3}=0,$
$\lim\limits_{|u|\rightarrow\infty}\frac{f(u)}{|u|^q}=0$ for some
$3<q<5$ and $\frac{f(u)}{u^3}$ is strictly increasing for $u>0$. They proved
that \eqref{1.2} has a positive ground state solution by using the
Nehari manifold.

Under the same condition $(V_5)$ on $V(x)$, Wang et al. in \cite{wt}
also proved the multiplicity of positive ground state solutions for
\eqref{1.2} by using the Nehari manifold when $N=3$ and
$f(x,u)=\lambda f(u)+|u|^4u$, which exhibits a critical growth,
where $f(u)=o(u^3)$, $f(u)u\geq0$, $\frac{f(u)}{u^3}$ is increasing
for $u>0$ and $|f(u)|\leq C(1+|u|^q)$ for some $q\in(3,5)$.

In \cite{lh}, Liu and He proved that \eqref{1.2} has infinitely many
solutions by using a variant version of fountain theorem when $N=3$,
$V(x)\in L^{\infty}_{loc}(\R^3)$ satisfies $(V_4)$ and
 $f(x,u)$ satisfies either $$(AR)~~\hbox{and}~~
\liminf\limits_{|u|\rightarrow\infty}\frac{F(x,u)}{|u|^\alpha}>0~\hbox{uniformly~in}~x\in\R^3~~\hbox{for~
some}~4<\alpha<6$$or
$$\lim\limits_{|u|\rightarrow\infty}\frac{F(x,u)}{|u|^4}=+\infty\hbox{~~and~~}\widetilde{F}(x,u)~\hbox{is~nondecreasing~for~all}~u>0,$$
where $\widetilde{F}(x,u)=\frac{1}{4}f(x,u)u-F(x,u)$.

Recently, in \cite{lls}, Li et al. studied the existence of a
positive solution for the following Kirchhoff problem
\begin{equation}\label{1.5}
\left(a+\varepsilon\ds\int_{\R^N}(|D u|^2+b|u|^2)\right)[-\Delta
u+bu]=f(u),~~\hbox{in}~\R^N,
\end{equation}
where $N\geq3$, $a,$ $b$ are positive constants, $\varepsilon\geq0$
is a parameter and the nonlinearity $f(u)$ satisfies the following
conditions:

$(H_1)$~~$f\in C(\R_+,\R_+)$ and $|f(u)|\leq C(|u|+|u|^{q-1})$ for
all $u\in\R_+$ and some $q\in (2,2^*)$, where $2^*=\frac{2N}{N-2}$;

$(H_2)$~~$\lim\limits_{u\rightarrow0}\frac{f(u)}{u}=0$;

$(H_3)$~~$\lim\limits_{u\rightarrow\infty}\frac{f(u)}{u}=\infty$.

\noindent By using a truncation argument combined with a
monotonicity trick introduced by Jeanjean \cite{j} (see also Struwe
\cite{s}), they showed that there exists $\varepsilon_0>0$ such that
for any $\varepsilon\in[0,\varepsilon_0)$, \eqref{1.5} has at least
one positive radial symmetric solution. However, their method could
be applied neither to the case that $\varepsilon$ is an arbitrary
positive constant nor to get a ground state solution in $H^1(\R^3)$.

Problem \eqref{1.1} is an important typical case for \eqref{1.2}
when $N=3$ and $f(x,u)=|u|^{p-1}u$ with $2<p<5$. For $2<p<5$, $f(x,u)$
may not be 4-superlinear at infinity, let alone $(AR)$. To the best
of our knowledge, the existence of nontrivial solutions was proved
only for $3<p<5$ (see e.g. \cite{hz}) and there is no existence
result for nontrivial solutions to \eqref{1.1} when $2<p\leq3$. The
difficulty is to get a bounded $(PS)$ sequence and to prove that the $(PS)$
sequence weakly converges to a critical point of the corresponding functional in $H^1(\R^3)$.

Motivated by the works described above, particularly, by the results
in \cite{hz,lls}, we try to get the existence of positive ground
state solutions to \eqref{1.1}. To state our main result, suppose
that $V(x)$ satisfies $(V_1)-(V_3)$ and $a>0$ is fixed, we introduce
an equivalent norm on $H^1(\R^3)$: the norm of $u\in H^1(\R^3)$ is
defined as
$$\|u\|=\left(\ds\int_{\R^3}a|Du|^2+V(x)|u|^2\right)^{\frac{1}{2}},$$which
is induced by the corresponding inner product on $H^1(\R^3)$. Weak
solutions to \eqref{1.1} correspond to critical points of the
following functional
\begin{equation}\label{1.10}
I_V(u)=\frac{1}{2}\ds\int_{\R^3}(a|Du|^2+V(x)|u|^2)+\frac{b}{4}\left(\ds\int_{\R^3}|Du|^2\right)^2-\frac{1}{p+1}\ds\int_{\R^3}|u|^{p+1}.
\end{equation}
We mention that although $I_V(u)$ is well defined in $H^1(\R^3)$ for
$1<p<5$, there exists a nontrivial solution to \eqref{1.1} only if
$2<p<5$ (see Theorem 1.1 below). We say a nontrivial weak solution
$u$ to \eqref{1.1} a ground state solution if $I_V(u)\leq I_V(w)$
for any nontrivial solution $w$ to
\eqref{1.1}.\\

Our main result is as follows:

\begin{thm}\label{thm1.1}\ \ If $V(x)$ satisfies $(V_1)-(V_3)$, then problem \eqref{1.1} has a positive
ground state solution for any $2<p<5$.
\end{thm}

\begin{rem}\label{rem1.2}\ \  These hypotheses $(V_1)-(V_3)$ on $V(x)$
above were introduced to study the Schr\"{o}dinger-Poisson system in
\cite{zz} and have physical meaning. There are indeed functions
which satisfy $(V_1)-(V_3)$. For example,
$V(x)=V_1-\frac{1}{|x|+1}$, where $V_1>1$ is a positive constant.
\end{rem}

Theorem 1.1 can be viewed as a partial extension of a main result in
\cite{hz} and extends the main result in \cite{zz} to the Kirchhoff
equation.

Now we give our main idea for the proof of Theorem 1.1. Since $(AR)$
or 4-superlinearity does not hold, the functional $I_V$ does not
always possess a mountain-pass geometry. Moreover, since $2<p<5$, it
is difficult to get the boundedness of any $(PS)$ sequence even if a
$(PS)$ sequence has been obtained. To overcome this difficulty,
inspired by \cite{lls,zz}, we use an indirect approach developed by
Jeanjean. We apply the following proposition due to Jeanjean
\cite{j}.
\begin{prop}\label{pro1.3}(\cite{j}, Theorem 1.1)\ \
Let $(X,\|\cdot\|)$ be a Banach space and $T\subset\R_+$ be an
interval. Consider a family of $C^1$ functionals on $X$ of the form
$$\Phi_\lambda(u)=A(u)-\lambda B(u),~~\forall~\lambda\in T,$$
with $B(u)\geq0$ and either $A(u)\rightarrow+\infty$ or
$B(u)\rightarrow+\infty$ as $\|u\|\rightarrow+\infty$. Assume that
there are two points $v_1$, $v_2\in X$ such that
$$c_\lambda=\inf_{\gamma\in
\Gamma}\max_{t\in[0,1]}\Phi_\lambda(\gamma(t))>\max\{\Phi_\lambda(v_1),~\Phi_\lambda(v_2)\},~~\forall~\lambda\in
T,$$ where
$$\Gamma=\{\gamma\in C([0,1],X)|~\gamma(0)=v_1,~\gamma(1)=v_2\}.$$
Then, for almost every $\lambda\in T$, there is a bounded
$(PS)_{c_\lambda}$ sequence in $X$.
\end{prop}

Let $T=[\delta,1],$ where $\delta\in(0,1)$ is a positive constant.
We consider a family of $C^1$ functionals defined by
$$I_{V,\lambda}(u)=\frac{1}{2}\ds\int_{\R^3}(a|Du|^2+V(x)|u|^2)+\frac{b}{4}\left(\ds\int_{\R^3}|Du|^2\right)^2-\frac{\lambda}{p+1}\ds\int_{\R^3}|u|^{p+1},~~\forall~\lambda\in [\delta,1].$$
By $(V_2)$ and Proposition 1.3, for a.e. $\lambda\in [\delta,1]$, there
exists a bounded $(PS)_{c_\lambda}$ sequence in $H^1(\R^3)$, denoted by $\{u_n\}$, where
$c_\lambda$ is given below (see Lemma 3.1). We can not easily see
that $I^\prime_{V,\lambda}$ is weakly sequentially continuous in
$H^1(\R^3)$ by direct calculations due to the existence of the
nonlocal term $\int_{\R^3}|Du|^2$. Indeed, in general, we do not know $\int_{\R^3}|Du_n|^2\rightarrow\int_{\R^3}|Du|^2$ from $u_n\rightharpoonup u$ in $H^1(\R^3)$. For problem \eqref{1.2}, this difficulty was overcome in \cite{lls}\cite{jw} when $V(x)\equiv \hbox{const}$ by using the radially symmetric Sobolev space $H^1_r(\R^3)=\{u\in H^1(\R^3)|~u(|x|)=u(x)\}$, where the embeddings $H^1_r(\R^3)\hookrightarrow L^q(\R^3)$ ($2<q<6$) are compact. If $V(x)$ satisfies $(V_4)$, this difficulty was dealt with in \cite{lh}\cite{w} by using the weighted Sobolev space $E=\{u\in H^1(\R^3)|~\int_{\R^3}V(x)|u|^2<\infty\}$ to guarantee that $(PS)$ condition holds. In \cite{hz} and \cite{wt}, $V(x)$ satisfies $(V_5)$, then the method used in \cite{lls}\cite{jw}\cite{lh}\cite{w} can not work. However, for the mountain-pass level $c$, it can be proved that each $(PS)_c$ sequence weakly converges to a critical point of the corresponding functional in $H^1(\R^3)$. Their argument strongly depends on the fact that $c=\inf\Psi(N)$, where $N=\{u\in H^1(\R^3)\backslash\{0\}|~\langle\Psi^\prime(u),u\rangle=0\}$ and $\frac{f(u)}{u^3}$ is strictly increasing for $u>0$. As we deal
with problem \eqref{1.1} in $H^1(\R^3)$, the Sobolev embeddings
$H^1(\R^3)\hookrightarrow L^q(\R^3),$ $q\in[2,2^*)$ are not compact. The nonlinearity $|u|^{p-1}u$ with $p\in(2,5)$ implies that the monotonicity of $\frac{|u|^{p-1}u}{u^3}$ does not always hold. So the arguments mentioned above can not be applied here to get a critical point of $I_{V,\lambda}$ from the bounded $(PS)_{c_\lambda}$ sequence $\{u_n\}$. To overcome this difficulty, although we can not directly prove that the weak limit $u\in H^1(\R^3)$ of $\{u_n\}$ is a critical point of $I_{V,\lambda}$, but we do easily see that $u$ is a critical point of the following functional
$$J_{V,\lambda}(u)=\frac{a+bA^2}{2}\ds\int_{\R^3}|Du|^2+\frac{1}{2}\ds\int_{\R^3}V(x)|u|^2-\frac{\lambda}{p+1}\ds\int_{\R^3}|u|^{p+1}$$
and $\{u_n\}$ is a $(PS)_{c_\lambda+\frac{bA^4}{4}}$ sequence for $J_{V,\lambda}$,
where $A^2=\lim\limits_{n\rightarrow\infty}\int_{\R^3}|Du_n|^2.$
By $(V_2)$, we try to establish a version of global compactness lemma (see Lemma 3.4 below) related to the functional $J_{V,\lambda}$ and its limited functional $$J^\infty_\lambda(u)=\frac{a+bA^2}{2}\ds\int_{\R^3}|Du|^2+\frac{1}{2}\ds\int_{\R^3}V_\infty|u|^2-\frac{\lambda}{p+1}\ds\int_{\R^3}|u|^{p+1}.$$
To apply the global compactness
lemma, first of all, we need to consider the existence of ground
state solutions of the associated "limit problem" of \eqref{1.1},
which is given as
\begin{equation}\label{1.13}
\left\{%
\begin{array}{ll}
    -\left(a+b\ds\int_{\R^3}|D u|^2\right)\Delta u+V_\infty u=\lambda|u|^{p-1}u, & \hbox{$x\in \R^3$}, \\
    u\in H^1(\R^3),~~~~u>0, & \hbox{$x\in \R^3$}\\
\end{array}%
\right.\end{equation} and their corresponding least energy of the
associated limited functional
$$
I^\infty_{\lambda}(u)=\frac{1}{2}\ds\int_{\R^3}(a|Du|^2+V_\infty|u|^2)+\frac{b}{4}\left(\ds\int_{\R^3}|Du|^2\right)^2-\frac{\lambda}{p+1}\ds\int_{\R^3}|u|^{p+1}.
$$

We obtain the following result:

\begin{thm}\label{thm1.4}
\eqref{1.13} has a positive ground state solution in $H^1(\R^3)$ for
all $2<p<5$.
\end{thm}

\begin{rem}\label{rem1.5}\ \
Theorem 1.4 in this paper is different from the main result in
\cite{lls}, since the equation \eqref{1.13} is different from the
equation \eqref{1.5} and we prove the existence of a positive ground
state solution in $H^1(\R^3)$.
\end{rem}

Therefore, by using Theorem 1.4 and applying the global compactness lemma and conditions $(V_1)$ $(V_2)$, we can prove that $(PS)_{c_\lambda}$ condition holds. During the proof, more careful analysis is needed to consider the relationship between $J_\lambda^\infty(w)$ and the least energy of $I_\lambda^\infty$, where $w$ is any critical point of $J_\lambda^\infty$ obtained in the global compactness lemma. Finally, choosing a
sequence $\{\lambda_n\}\subset [\delta,1]$ with $\lambda_n\rightarrow1$,
there exists a sequence of nontrivial weak solutions
$\{u_{\lambda_n}\}\subset H^1(\R^3)$. We can prove that
$\{u_{\lambda_n}\}$ is a bounded $(PS)_{c_1}$ sequence for
$I_{V}=I_{V,1}$ by using the
Poho\u{z}aev identity and $(V_1)$, which yields Theorem 1.1.\\

As for problem \eqref{1.13}, $I^\infty_\lambda$ does not always
satisfy $(PS)$ condition and it is difficult to get a ground state
solution even if nontrivial weak critical points for
$I^\infty_\lambda$ have been obtained since $2<p<5$. For simplicity,
we may assume that $V_\infty=\lambda\equiv1$ in \eqref{1.13}, i.e.
\begin{equation}\label{1.16}
\left\{%
\begin{array}{ll}
    -\left(a+b\ds\int_{\R^3}|D u|^2\right)\Delta u+u=|u|^{p-1}u, & \hbox{$x\in \R^3$}, \\
    u\in H^1(\R^3),~~~~u>0, & \hbox{$x\in \R^3$}\\
\end{array}%
\right.\end{equation} and denote the corresponding functional $I(u)$
instead of $I^\infty_\lambda(u)$, i.e.
$$
I(u)=\frac{1}{2}\ds\int_{\R^3}(a|Du|^2+|u|^2)+\frac{b}{4}\left(\ds\int_{\R^3}|Du|^2\right)^2-\frac{1}{p+1}\ds\int_{\R^3}|u|^{p+1}.
$$
To obtain Theorem 1.4, we need to prove that problem \eqref{1.16}
has a positive ground state solution in $H^1(\R^3)$ for all $2<p<5$.

In the last decade, the following nonlinear Schr\"{o}dinger-Poisson
system
\begin{equation}\label{1.7}\left\{%
\begin{array}{ll}
    -\Delta u+u+ \lambda\phi u=|u|^{p-1}u, & \hbox{$x\in \R^3$}, \\
    -\Delta \phi=u^2, & \hbox{$x\in \R^3$},\\
\end{array}%
\right.
\end{equation} where $\lambda>0$ is a parameter and $1<p<5$, has been extensively studied, see e.g. \cite{aru,ap,c,dm,ru}.
$(u,\phi)\in H^1(\R^3)\times D^{1,2}(\R^3)$ is a weak solution to
\eqref{1.7} if $u$ is a critical point of the functional
\begin{equation}\label{1.8}
E(u)=\frac{1}{2}\ds\int_{\R^3}(|Du|^2+u^2)+\frac{1}{4}\lambda\ds\int_{\R^3}\phi_uu^2
-\frac{1}{p+1}\ds\int_{\R^3}|u|^{p+1},
\end{equation}where
$\phi_u$ is the unique solution of the second equation in
\eqref{1.7}. Whether there is a nontrivial solution to \eqref{1.7}
or not depends on the range of the parameter $\lambda$ and $p$. For
$p\geq3$, it is easy to prove that the energy functional $E(u)$ satisfies the (PS) condition and one can use the mountain-pass theorem to get the existence of a nontrivial
solution to \eqref{1.7} (see\cite{c,dm}). But for
$p\in (2,3)$, the method in \cite{c,dm} can not be applied. In
\cite{ru}, Ruiz proved that when $1<p\leq2$, \eqref{1.7} has at
least two nontrivial solutions for small $\lambda$ by using the
mountain-pass theorem and Ekeland's variational principle and
\eqref{1.7} has no nontrivial solution if $\lambda\geq\frac{1}{4}$.
To deal with the case when $2<p<3$, a constrained minimization
method was used. It was proved in \cite{ru} that there is a positive
radial nontrivial solution to \eqref{1.7} for $2<p<5$. However, the
traditional method which takes the minimum of the functional on its
Nehari manifold does not work. In \cite{ru}, the constrained
minimization was carried out on a new manifold
$\overline{\mathcal{M}}$, which is obtained by combining the usual
Nehari manifold and the Poho\u{z}ave identity of \eqref{1.7} proved
in \cite{dm}. In fact,
$$
\overline{\mathcal{M}}=\{u\in
H^1_r(\R^3)\backslash\{0\}:~\overline{G}(u)=0\},
$$ where $H^1_r(\R^3)$ denotes the subspace of radially symmetric functions
in $H^1(\R^3)$ and
\begin{equation}\label{1.9}
\overline{G}(u)=2\langle E^\prime(u),u\rangle-\overline{P}(u)
\end{equation}
 and $\overline{P}(u)=0$ is the Poho\u{z}ave identity, i.e.
$$\overline{P}(u)=\frac{1}{2}\ds\int_{\R^3}|Du|^2+\frac{3}{2}\ds\int_{\R^3}u^2+\frac{5}{4}\lambda\ds\int_{\R^3}\phi_uu^2
-\frac{{3}}{p+1}\ds\int_{\R^3}|u|^{p+1}.$$ In \cite{ap}, Azzollini
and Pomponio used the same manifold as in \cite{ru} and the
concentration-compactness argument to prove the existence of
positive ground state solutions to \eqref{1.7} when $2<p<5$ and
$\lambda=1$.

Motivated by \cite{ap,ru}, we try to use the constrained
minimization on a manifold to prove Theorem 1.4. The main difficulty
is to choose a suitable manifold. As we describe before, the usual
Nehari manifold is not suitable because it is difficult to prove the
boundedness of the minimizing sequence. So we follow \cite{ru} to
take the minimum on a new manifold, which is obtained by combining
the Nehari manifold and the corresponding Poho\u{z}ave type
identity: for any solution $u\in H^1(\R^3)$ to \eqref{1.16},
$$P(u)\triangleq\frac{1}{2}a\ds\int_{\R^3}|Du|^2+\frac{3}{2}\ds\int_{\R^3}|u|^2+\frac{1}{2}b\left(\ds\int_{\R^3}|Du|^2\right)^2-\frac{3}{p+1}\ds\int_{\R^3}|u|^{p+1}=0,$$
which will be proved in section 2 (see Lemma 2.1). In fact, the
manifold we use is defined by
\begin{equation}\label{1.12}
\mathcal{M}\triangleq\{u\in H^1(\R^3)\backslash\{0\}:~G(u)=0\},
\end{equation} where
$$G(u)=\langle I^\prime(u),u\rangle+P(u).$$

 Our choice of
$\mathcal{\mathcal{M}}$ is slightly different from that in
\cite{ru}, which is
$$\overline{\mathcal{M}}=\{u\in H^1_r(\R^3)\backslash\{0\}:~2\langle
E^\prime(u),u\rangle-\overline{P}(u)=0\}.$$ The reason is that if we
chose $\overline{\mathcal{M}}$ instead of $\mathcal{M}$, we would
face the difficulty to prove the boundedness of the minimizing
sequence. Our idea to get $\mathcal{M}$ is similar to that of
\cite{ru} and can be described as follows. For $u\in
H^1(\R^3)\backslash \{0\}$, let $\alpha$, $\beta\in\R$ be constants
and $u_t(x)=t^\alpha u(t^\beta x)$, $t>0$, since $2<p<5$,
$$
\begin{array}{ll}
I(u_t)&=
\frac{at^{2\alpha-\beta}}{2}\ds\int_{\R^3}|Du|^2+\frac{t^{2\alpha-3\beta}}{2}\ds\int_{\R^3}|u|^2+\frac{bt^{4\alpha-2\beta}}{4}\left(\ds\int_{\R^3}|Du|^2\right)^2-\frac{t^{\alpha(p+1)-3\beta}}{p+1}\ds\int_{\R^3}|u|^{p+1}\\[5mm]
&\rightarrow-\infty~~\hbox{as}~~t\rightarrow+\infty~~~~~\hbox{if}~~~~\alpha+\beta=0,~\alpha>0.
\end{array}
$$
So take $\alpha=1,$ $\beta=-1$, then the function
$\gamma(t)\triangleq I(u_t)$ would have a unique critical point
$t_0>0$ corresponding to its maximum (see Lemma 2.3). Moreover, if
$u$ is a solution of \eqref{1.16}, then $t_0=1$ and hence
$\gamma^\prime(1)=0$, i.e.
$$G(u)\triangleq\frac{3}{2}a\ds\int_{\R^3}|Du|^2+\frac{5}{2}\ds\int_{\R^3}|u|^2+\frac{3}{2}b\left(\ds\int_{\R^3}|Du|^2\right)^2-\frac{p+4}{p+1}\ds\int_{\R^3}|u|^{p+1}=0.
$$
We easily see that
\begin{equation}\label{1.15}
G(u)=\langle I^\prime(u),u\rangle+P(u),\end{equation} which gives
the clue to define $\mathcal{M}$. Although we mainly follow the
procedure of \cite{ru}, as we consider ground state solutions, we
have to work in $H^1(\R^3)$ as in \cite{ap} instead of
$H^1_r(\R^3)$, which results in that the method used in \cite{lls}
can not be applied. So the compactness of the minimizing sequence is
handled by using concentration-compactness principle, which is much
more complicated than using $H^1_r(\R^3)$.\\

We also obtain a supplementary result to Theorem 1.1 in \cite{lls}
in a special case $f(u)=|u|^{p-1}u$, where $1<p\leq2$. We consider
the non-existence about the following Kirchhoff type problem
\begin{equation}\label{1.11}
\left\{%
\begin{array}{ll}
    \left(a+b\lambda\ds\int_{\R^3}(|D u|^2+V(x)|u|^2)\right)[-\Delta u+V(x)u]=|u|^{p-1}u, & \hbox{$x\in \R^3$}, \\
    u\in H^1(\R^3),\\
\end{array}%
\right.\end{equation}where $\lambda>0$ is a parameter, $a>1,$ $b>0$
are constants.

\begin{thm}\label{thm1.6}\ \ Let $1<p\leq2$, $a>1$, $b>0$
be constants and $V(x)$ either satisfy $(V_2)(V_3)$ or be a
positive constant, then there exists
$\lambda_0=\frac{1}{4b(a-1)C^3}>0$ such that for any $\lambda\geq
\lambda_0$, \eqref{1.11} has no nontrivial solution, where $C$ is
the best Sobolev constant for the embedding from $H^1(\R^3)$ into
$L^3(\R^3)$, i.e. $C=\inf\limits_{u\in
H^1(\R^3)\backslash\{0\}}\frac{\int_{\R^3}(|Du|^2+V(x)u^2)}{|u|_3^2}.$
\end{thm}

Theorem 1.6 is a related result to the main result in \cite{lls}.
However, \cite{lls} did not give such a non-existence result.

The paper is organized as follows. In $\S$ 2, we present some
preliminary results. In $\S$ 3, we will prove our main results
Theorem 1.4 and Theorem 1.1. In $\S$ 4, we give the proof of Theorem
1.6.

\section{Preliminary  Results }

In this section, we give some preliminary results.

\begin{lem}\label{lem2.1}(Poho\u{z}ave Identity)\ \ Assume $V(x)$
satisfies $(V_1)-(V_3)$. Let $u\in H^1(\R^3)$ be a weak solution to
problem \eqref{1.1} and $p\in(1,5)$, then we have the following
Poho\u{z}aev identity:
\begin{equation}\label{2.1}
\frac{a}{2}\ds\int_{\R^3}|Du|^2+\frac{3}{2}\ds\int_{\R^3}V(x)|u|^2+\frac{1}{2}\ds\int_{\R^3}(DV(x),x)|u|^2+\frac{b}{2}\left(\ds\int_{\R^3}|Du|^2\right)^2-\frac{3}{p+1}\ds\int_{\R^3}|u|^{p+1}=0.
\end{equation}
\end{lem}
\begin{proof}~~
The proof is standard, so we omit it (see e.g. \cite{bl,dm}).
\end{proof}
For the case when $V\equiv 1$, the Poho\u{z}aev identity can be
rewritten as follows:
\begin{equation}\label{2.2}
P(u)\triangleq\frac{1}{2}a\ds\int_{\R^3}|Du|^2+\frac{3}{2}\ds\int_{\R^3}|u|^2+\frac{1}{2}b\left(\ds\int_{\R^3}|Du|^2\right)^2-\frac{3}{p+1}\ds\int_{\R^3}|u|^{p+1}=0.
\end{equation}

\begin{lem}\label{lem2.2}\ \
Let $p\in(2,5)$, then $I$ is not bounded from below.
\end{lem}
\begin{proof}~~
For any $u\in H^1(\R^3)\backslash\{0\}$, set $u_t(x)=tu(t^{-1}x)$,
$t>0$. Then
$$I(u_t)=\frac{a}{2}t^3\ds\int_{\R^3}|Du|^2+\frac{1}{2}t^5\ds\int_{\R^3}|u|^2+\frac{b}{4}t^6\left(\ds\int_{\R^3}|Du|^2\right)^2-\frac{1}{p+1}t^{p+4}\ds\int_{\R^3}|u|^{p+1}.$$
Since $p\in(2,5)$, we see that $I(u_t)\rightarrow-\infty$ as
$t\rightarrow +\infty$.
\end{proof}
Lemma 2.2 shows that $I$ possesses a mountain pass geometry around
$0\in H^1(\R^3)$. As we mentioned in  $\S$1, $I$ satisfies $(PS)$
condition for $3<p<5$, hence the existence of at least one
nontrivial solution can be obtained. However, for $2<p\leq3$, we
need to consider the constrained minimization on a suitable manifold
as \cite{ru} did.

To motivate the definition of such a manifold, we need the following
lemmas.
\begin{lem}\label{2.3}\ \
Let $C_i$ $(i=1,2,3,4)$ be positive constants and $p>2$. If
$f(t)=C_1 t^3+C_2 t^5+C_3 t^6-C_4 t^{p+4}$ for $t\geq0$. Then $f$
has a unique critical point which corresponds to its maximum.
\end{lem}
\begin{proof}~~
The proof is similar to that of Lemma 3.3 in \cite{ru} and is
elementary. We omit the proof.
\end{proof}
Suppose that $u\in H^1(\R^3)$ is a nontrivial critical point of $I$
and $u_t(x)=tu(t^{-1}x)$ for $t>0$. Set
$$\gamma(t)\triangleq I(u_t)=\frac{a}{2}t^3\ds\int_{\R^3}|Du|^2+\frac{1}{2}t^5\ds\int_{\R^3}|u|^2
+\frac{b}{4}t^6\left(\ds\int_{\R^3}|Du|^2\right)^2-\frac{1}{p+1}t^{p+4}\ds\int_{\R^3}|u|^{p+1}.$$
By Lemma 2.3, $\gamma$ has a unique critical point $t_0>0$
corresponding to its maximum. Since $u$ is a solution to
\eqref{1.16}, we see that $t_0=1$ and $\gamma^\prime(1)=0$, which
implies that
\begin{equation}\label{2.3}
G(u)\triangleq\frac{3}{2}a\ds\int_{\R^3}|Du|^2+\frac{5}{2}\ds\int_{\R^3}|u|^2+
\frac{3}{2}b\left(\ds\int_{\R^3}|Du|^2\right)^2-\frac{p+4}{p+1}\ds\int_{\R^3}|u|^{p+1}=0.
\end{equation}
So we define
$$\mathcal{M}=\{u\in H^1(\R^3)\backslash\{0\}|~G(u)=0\}.$$
It is clear that
\begin{equation}\label{2.4}
G(u)=\langle I^\prime(u),u\rangle+P(u),
\end{equation}
where $P(u)$ is given in \eqref{2.2}.

\begin{rem}\label{rem2.4}\ \
If $u\in H^1(\R^3)$ is a nontrivial weak solution to \eqref{1.16},
then by Lemma 2.1 and \eqref{2.4}, we see that $u\in \mathcal{M}$.
Our definition of $\mathcal{M}$ is slightly different from that of
\cite{ru}.
\end{rem}

\begin{lem}\label{lem2.5}\ \
For any $u\in H^1(\R^3)\backslash\{0\}$, there is a unique
$\tilde{t}>0$ such that $u_{\tilde{t}}\in\mathcal{M}$, where
$u_{\tilde{t}}(x)=\tilde{t}u(\tilde{t}^{-1}x)$. Moreover,
$I(u_{\tilde{t}})=\max\limits_{t>0}I(u_t).$
\end{lem}
\begin{proof}~~
 For any $u\in H^1(\R^3)\backslash\{0\}$ and $t>0$, set $u_t(x)=tu(t^{-1}x)$.
Consider
$$\gamma(t)\triangleq I(u_t)=\frac{a}{2}t^3\ds\int_{\R^3}|Du|^2+\frac{1}{2}t^5\ds\int_{\R^3}|u|^2
+\frac{b}{4}t^6\left(\ds\int_{\R^3}|Du|^2\right)^2-\frac{1}{p+1}t^{p+4}\ds\int_{\R^3}|u|^{p+1}.$$
By Lemma 2.3, $\gamma$ has a unique critical point $\tilde{t}>0$
corresponding to its maximum. Then
$\gamma(\tilde{t})=\max\limits_{t>0}\gamma(t)$ and
$\gamma^\prime(\tilde{t})=0$. So
$$\frac{3}{2}a\tilde{t}^2\ds\int_{\R^3}|Du|^2+\frac{5}{2}\tilde{t}^4\ds\int_{\R^3}|u|^2+
\frac{3}{2}b\tilde{t}^5\left(\ds\int_{\R^3}|Du|^2\right)^2-\frac{p+4}{p+1}\tilde{t}^{p+3}\ds\int_{\R^3}|u|^{p+1}=0,$$
then $G(u_{\tilde{t}})=0$ and $u_{\tilde{t}}\in \mathcal{M}$.
\end{proof}

\begin{lem}\label{lem2.6}\ \
Suppose that $p\in(2,5)$, then $\mathcal{M}$ is a natural
$C^1$-manifold and every critical point of $I|_{\mathcal{M}}$ is a
critical point of $I$ in $H^1(\R^3)$.
\end{lem}
\begin{proof}~~
To prove the lemma, we follow the argument used in \cite{ru}, which
deals with the Schr\"{o}dinger-Poisson system. By Lemma 2.5,
$\mathcal{M}\neq\varnothing$. The proof consists of four steps.

\noindent\textbf{Step 1}.~~$0\notin \partial \mathcal{M}$.

By the Sobolev embedding inequality, choosing $r>0$ small enough,
then there exist $\rho>0$, $C>0$ such that
$$\frac{3}{2}a\ds\int_{\R^3}|Du|^2+\frac{5}{2}\ds\int_{\R^3}|u|^2+
\frac{3}{2}b\left(\ds\int_{\R^3}|Du|^2\right)^2-\frac{p+4}{p+1}\ds\int_{\R^3}|u|^{p+1}
\geq \frac{3}{2}\|u\|^2-C\frac{p+4}{p+1}\|u\|^{p+1}>\rho$$ for
$\|u\|=r$ small. Then $0\notin \partial\mathcal{M}$.

\noindent\textbf{Step 2}.~~$I(u)>0$ for all $u\in\mathcal{M}$.

For any $u\in \mathcal{M}$, let $k\triangleq I(u)$ and
$$\alpha\triangleq a\ds\int_{\R^3}|Du|^2,~~~
\beta\triangleq \ds\int_{\R^3}|u|^2,~~~\mu\triangleq
b\left(\ds\int_{\R^3}|Du|^2\right)^2,
~~~\delta\triangleq\ds\int_{\R^3}|u|^{p+1}.$$ Then $\alpha,$
$\beta,$ $\mu,$ $\delta$ are positive and
$$
\left\{%
\begin{array}{ll}
    \frac{1}{2}\alpha+\frac{1}{2}\beta+\frac{1}{4}\mu-\frac{1}{p+1}\delta=k, \\
 \frac{3}{2}\alpha+\frac{5}{2}\beta+\frac{3}{2}\mu-\frac{p+4}{p+1}\delta=0, \\
\end{array}%
\right.
$$
hence
$$\mu=\frac{4k(p+4)-2\alpha(p+1)-2\beta(p-1)}{p-2},~~~\delta=\frac{6k-\frac{3\alpha}{2}-\frac{\beta}{2}}{p-2}(p+1).$$
Since $\mu>0$ and $p>2$, we must have
$$(\alpha+\beta)(p-1)<\beta(p-1)+\alpha(p+1)<2k(p+4).$$
Thus $I(u)=k>\frac{(\alpha+\beta)(p-1)}{2(p+4)}>0$.

\noindent\textbf{Step 3}.~~$G^\prime(u)\neq0$ for every
$u\in\mathcal{M}$, hence $\mathcal{M}$ is a $C^1$-manifold.

Just suppose that $G^\prime(u)=0$ for some $u\in\mathcal{M}$. In a
weak sense, the equation $G^\prime(u)=0$ can be written as
$$-3\left(a+2b\ds\int_{\R^3}|Du|^2\right)\Delta
u+5u=(p+4)|u|^{p-1}u.$$ So combining Lemma 2.1 and the notations
defined in \textbf{Step 2}, we have that
$$
\left\{%
\begin{array}{ll}
    \frac{1}{2}\alpha+\frac{1}{2}\beta+\frac{1}{4}\mu-\frac{1}{p+1}\delta=k>0, \\
 \frac{3}{2}\alpha+\frac{5}{2}\beta+\frac{3}{2}\mu-\frac{p+4}{p+1}\delta=0, \\
 3\alpha+5\beta+6\mu-(p+4)\delta=0,\\
 \frac{3}{2}\alpha+\frac{15}{2}\beta+3\mu-\frac{3(p+4)}{p+1}\delta=0. \\
\end{array}%
\right.
$$
Since $p\in(2,5)$, we can conclude that the above linear system has
a unique solution given as
$$\alpha=\frac{10k(p+4)}{3p},~~~~~~~\beta=\frac{2k(p-5)(p+4)}{p(p-1)},$$
$$\mu=-\frac{20k(p+4)}{3p},~~~~~~~\delta=-\frac{20k(p+1)}{p(p-1)}.$$
Since $u\neq0$ and $\mu$, $\delta$ are positive, we get a
contradition. Then $G^\prime(u)\neq0$ for every $u\in\mathcal{M}$
and by the Implicit Function theorem, $\mathcal{M}$ is a
$C^1$-manifold.

\noindent\textbf{Step 4}.~~Every critical point of
$I|_{\mathcal{M}}$ is a critical point of $I$ in $H^1(\R^3)$.

If $u$ is a critical point of $I|_{\mathcal{M}}$, i.e. $u\in
\mathcal{M}$ and $(I|_{\mathcal{M}})^\prime(u)=0$. There is a
Lagrange multiplier $\lambda\in\R$ such that $I^\prime(u)-\lambda
G^\prime(u)=0$. It is enough to show that $\lambda=0$.

The equation $I^\prime(u)=\lambda G^\prime(u)$ can be written, in a
weak sense, as
$$-\left(a+b\ds\int_{\R^3}|Du|^2\right)\Delta
u+u-|u|^{p-1}u=\lambda\left[-3\left(a+2b\ds\int_{\R^3}|Du|^2\right)\Delta
u+5u-(p+4)|u|^{p-1}u\right].$$ Hence $u$ solves the equation
\begin{equation}\label{2.5}
-(3\lambda-1)a\Delta u-(6\lambda-1)b\ds\int_{\R^3}|Du|^2\Delta
u+(5\lambda-1)u-[(p+4)\lambda-1]|u|^{p-1}u=0.
\end{equation}
Using the notations in \textbf{Step 2}, by Lemma 2.1 and \eqref{2.5}
we have that
\begin{equation}\label{2.6}
\left\{%
\begin{array}{ll}
    \frac{1}{2}\alpha+\frac{1}{2}\beta+\frac{1}{4}\mu-\frac{1}{p+1}\delta=k>0, \\
 \frac{3}{2}\alpha+\frac{5}{2}\beta+\frac{3}{2}\mu-\frac{p+4}{p+1}\delta=0, \\
 (3\lambda-1)\alpha+(5\lambda-1)\beta+(6\lambda-1)\mu-[(p+4)\lambda-1]\delta=0,\\
 \frac{3\lambda-1}{2}\alpha+\frac{3(5\lambda-1)}{2}\beta+\frac{6\lambda-1}{2}\mu-\frac{3[(p+4)\lambda-1]}{p+1}\delta=0, \\
\end{array}%
\right.
\end{equation}
which is a linear system for $\alpha$, $\beta$, $\mu$ and $\delta$.
The coefficient matrix of \eqref{2.6} is
$$A=\left(\begin{array}{cccc}
\frac{1}{2}&\frac{1}{2}&\frac{1}{4}&-\frac{1}{p+1}\\
\frac{3}{2}&\frac{5}{2}&\frac{3}{2}&-\frac{p+4}{p+1}\\
3\lambda-1&5\lambda-1&6\lambda-1&-[(p+4)\lambda-1]\\
\frac{3\lambda-1}{2}&\frac{3(5\lambda-1)}{2}&\frac{6\lambda-1}{2}&-\frac{3[(p+4)\lambda-1]}{p+1}
\end{array}\right)$$
 and its determinant is
$$\hbox{det} A=\frac{\lambda(p-1)(2p-1-9p\lambda)}{8(p+1)}.$$
Then $$\hbox{det} A=0\Longleftrightarrow
\lambda=0,~\lambda=\frac{2p-1}{9p},~p=1.$$ We will show that
$\lambda$ must be equal to zero by excluding the other two
possibilities:

(1)~~If $\lambda\neq0,$ $\lambda\neq\frac{2p-1}{9p},$ the linear
system \eqref{2.6} has a unique solution. We obtain the value of
$\beta$ and $\delta$ as follows:
$$\beta=-\frac{18k(p-5)[(p+4)\lambda-1]}{(p-1)(2p-1-9p\lambda)},~~~\delta=\frac{36k(1+p)(5\lambda-1)}{(p-1)(2p-1-9p\lambda)}.$$
Since $p\in(2,5)$, $\frac{1}{6}<\frac{2p-1}{9p}<\frac{1}{5}$. We
conclude that $\delta\leq0$ for $\lambda\in
[\frac{1}{5},+\infty)\cup(-\infty,\frac{2p-1}{9p})$ and $\beta<0$
for $\lambda\in(\frac{2p-1}{9p},\frac{1}{5}),$ however, this is
impossible since both $\delta$ and $\beta$ must be positive.

(2)~~If $\lambda=\frac{2p-1}{9p}.$ In such case, the latter two
equations in \eqref{2.6} are as follows
$$\left\{%
\begin{array}{ll}
 -\frac{p+1}{3p}\alpha+\frac{p-5}{9p}\beta+\frac{p-2}{3p}\mu-\frac{2(p+1)(p-2)}{9p}\delta=0,\\
 -\frac{p+1}{6p}\alpha+\frac{p-5}{6p}\beta+\frac{p-2}{6p}\mu-\frac{2(p-2)}{3p}\delta=0. \\
\end{array}%
\right.$$ Then
$$\beta+(p-2)\delta=0,$$
which is also impossible since $p>2$ and $\beta,$ $\delta$ must be
positive. Then $\lambda=0$, hence $I^\prime(u)=0$, i.e. $u$ is a
critical point of $I$.
\end{proof}

\begin{lem}\label{2.7}\ \
For $2<p<5$, there exists $C>0$ such that for any $u\in
\mathcal{M}$, $|u|_{p+1}\geq C$.
\end{lem}
\begin{proof}~~
For any $u\in\mathcal{M}$, $G(u)=0$. Since $2<p<5$, by the Sobolev
embedding inequality, there exists $C>0$ such that
$$
\begin{array}{ll}
0&=\frac{3}{2}a\ds\int_{\R^3}|Du|^2+\frac{5}{2}\ds\int_{\R^3}|u|^2+
\frac{3}{2}b\left(\ds\int_{\R^3}|Du|^2\right)^2-\frac{p+4}{p+1}\ds\int_{\R^3}|u|^{p+1}\\[5mm]
&\geq\frac{3}{2}\|u\|^2-\frac{p+4}{p+1}\ds\int_{\R^3}|u|^{p+1}\\[5mm]
&\geq\frac{3}{2}C|u|_{p+1}^2-\frac{p+4}{p+1}|u|_{p+1}^{p+1}.
\end{array}$$
Then
$|u|_{p+1}\geq\left[\frac{3C(p+1)}{2(p+4)}\right]^{\frac{1}{p-1}}.$
\end{proof}
Set
$$c_1=\inf\limits_{\eta\in
\Gamma}\max\limits_{t\in[0,1]}I(\eta(t)),~~~~c_2=\inf\limits_{u\in
H^1(\R^3)\backslash\{0\}}\max\limits_{t>0}I(u_t),~~~~c_3=\inf\limits_{u\in
\mathcal{M}}I(u),$$where $u_t(x)=tu(t^{-1}x)$ and $$\Gamma=\{\eta\in
C([0,1],H^1(\R^3))|~\eta(0)=0,~I(\eta(1))\leq0,~\eta(1)\neq0\}.$$

\begin{lem}\label{lem2.8}\ \
$c\triangleq c_1=c_2=c_3>0.$
\end{lem}
\begin{proof}~~The proof is similar to that of Proposition 3.11 in \cite{r}, where $\mathcal{M}$ was the Nehari
manifold. We give a detailed proof here for readers' convenience.

By Lemma 2.5, for each $u\in H^1(\R^3)\backslash\{0\}$, there exists
a unique $u_{\tilde{t}}\in \mathcal{M}$ such that
$$I(u_{\tilde{t}})=\max\limits_{t>0}I(u_t).$$
It follows that $c_2=c_3$.

For any $\eta\in \Gamma$, we claim that
$\eta([0,1])\cap\mathcal{M}\neq\varnothing$. Indeed, by \textbf{Step
1} in the proof of Lemma 2.6, we see that if $u\in
H^1(\R^3)\backslash\{0\}$ is interior to or on $\mathcal{M}$, then
$$\frac{3}{2}a\ds\int_{\R^3}|Du|^2+\frac{5}{2}\ds\int_{\R^3}|u|^2+
\frac{3}{2}b\left(\ds\int_{\R^3}|Du|^2\right)^2\geq\frac{p+4}{p+1}\ds\int_{\R^3}|u|^{p+1}$$
and
$$6I(u)> G(u)+\frac{1}{2}\|u\|^2+\frac{p-2}{p+1}\ds\int_{\R^3}|u|^{p+1}>0.$$
Hence $\eta$ crosses $\mathcal{M}$ since $\eta(0)=0$,
$I(\eta(1))\leq0$ and $\eta(1)\neq0$. Therefore
$$\max\limits_{t\in[0,1]}I(\eta(t))\geq \inf\limits_{u\in
\mathcal{M}}I(u)=c_3$$ and then $c_1\geq c_3$. On the other hand,
for $u\in H^1(\R^3)\backslash\{0\}$, by Lemma 2.2, $I(u_{t_0})<0$
for $t_0$ large enough. Set
$$\widetilde{\eta}(t)\triangleq \left\{%
\begin{array}{ll}
    u_{tt_0}, & \hbox{$t>0$}, \\
    0, & \hbox{$t=0$},\\
\end{array}%
\right.$$ then $\widetilde{\eta}\in \Gamma$. Hence
$$\max\limits_{t>0}I(u_{tt_0})\geq\max\limits_{t\in[0,1]}I(u_{tt_0})\geq\inf\limits_{\eta\in
\Gamma}\max\limits_{t\in[0,1]}I(\eta(t))=c_1.$$ So $c_2\geq c_1.$
\end{proof}
 By Lemma 2.6, Lemma 2.8 and Remark 2.4, if $u\in \mathcal{M}$ such that $I(u)=c$, then $u$ is a ground state solution to
\eqref{1.16}. So we look for critical points of $I$ restricted on
$\mathcal{M}$.

The following concentration-compactness principle is due to P. L.
Lions.
\begin{lem}\label{lem 2.9}(\cite{pl}, Lemma 1.1)\ \
Let $\{\rho_n\}$ be a sequence of nonnegative $L^1$ functions on
$\R^N$ satisfying $\ds\int_{\R^N}\rho_n=\lambda$, where $\lambda>0$
is fixed. There exists a subsequence, still denoted by $\{\rho_n\}$
satisfying one of the following three possibilities:

(i)~(Vanishing)~~ for all $R>0$, it holds
$$\lim\limits_{n\rightarrow+\infty}\sup\limits_{y\in
\R^N}\ds\int_{B_R(y)}\rho_n=0;$$

(ii)~(Compactness)~~ there exists $\{y_n\}\subset\R^N$ such that,
for any $\varepsilon>0,$ there exists an $R>0$ satisfying
$$\liminf_{n\rightarrow+\infty}\ds\int_{B_R(y_n)}\rho_n\geq
\lambda-\varepsilon;$$

(iii)~(Dichotomy)~~ there exists an $\alpha\in (0,\lambda)$ and
$\{y_n\}\subset\R^N$ such that for any $\varepsilon>0$, $\exists$
$R>0$, for all $r\geq R$ and $r^\prime\geq R$, it holds
$$\limsup_{n\rightarrow+\infty}\left|\alpha-\ds\int_{B_{r}(y_n)}\rho_n\right|+\left|(\lambda-\alpha)-\ds\int_{\R^N\backslash B_{r^\prime}(y_n)}\rho_n\right|<\varepsilon.$$
\end{lem}

\begin{lem}\label{lem2.10} (\cite{wi}, Lemma
1.21)\ \ Let $r>0$ and $2\leq q<2^*$. If $\{u_n\}$ is bounded in
$H^1(\R^N)$ and
$$\sup\limits_{y\in\R^N}\ds\int_{B_r(y)}|u_n|^q\rightarrow0,~~n\rightarrow+\infty,$$
then $u_n\rightarrow0$ in $L^s(\R^N)$ for $2<s<2^*$.
\end{lem}

\begin{lem}\label{lem2.11}\ \
Let $\{u_n\}\subset \mathcal{M}$ be a minimizing sequence for $c$,
which was given in Lemma 2.8. Then there exists $\{y_n\}\subset\R^3$
such that for any $\varepsilon>0$, there exists an $R>0$ satisfying
$$\ds\int_{\R^3\backslash B_R(y_n)}a|Du_n|^2+u_n^2\leq
\varepsilon.$$
\end{lem}
\begin{proof}~~
Suppose that $\{u_n\}\subset \mathcal{M}$ satisfying
\begin{equation}\label{2.8}
\lim\limits_{n\rightarrow+\infty}I(u_n)=c>0.
\end{equation}
We introduce a new functional $\Phi:H^1(\R^3)\rightarrow\R$ as
follows:
\begin{equation}\label{2.9}
\Phi(u)=\ds\int_{\R^3}\frac{a}{4}|Du|^2+\frac{1}{12}u^2+\frac{p-2}{6(p+1)}|u|^{p+1}.
\end{equation}
For $\forall~u\in \mathcal{M}$, we have that $I(u)=\Phi(u)\geq0$ and
then $\lim\limits_{n\rightarrow+\infty}\Phi(u_n)=c$. Hence $\{u_n\}$
is bounded in $H^1(\R^3)$. Up to a subsequence, we may assume that
there exists a $u\in H^1(\R^3)$ such that
\begin{equation}\label{2.10}
\left\{%
\begin{array}{ll}
u_n\rightharpoonup u,~~~~& \hbox{$ \,\ \hbox{in}~~H^1(\R^3)$},\\
u_n\rightarrow u,~~~~& \hbox{$ \,\
\hbox{in}~~L^s_{loc}(\R^3),~~\forall~s\in[1,6)$}.
\end{array}%
\right.
\end{equation}
To prove this theorem, we apply the concentration-compactness
principle Lemma 2.9. Set
\begin{equation}\label{2.11}
\rho_n\triangleq\frac{a}{4}|Du_n|^2+\frac{1}{12}u_n^2+\frac{p-2}{6(p+1)}|u_n|^{p+1},
\end{equation}
 then $\{\rho_n\}$ is a sequence of nonnegative $L^1$ functions on
$\R^3$ satisfying
$$\ds\int_{\R^3}\rho_n=\Phi(u_n)\rightarrow c>0.$$
By Lemma 2.9, there are three possibilities:

\noindent Vanishing:~~ for all $R>0$, it holds
$$\lim\limits_{n\rightarrow+\infty}\sup\limits_{y\in
\R^3}\ds\int_{B_R(y)}\rho_n=0;$$

\noindent Compactness:~~ there exists $\{y_n\}\subset\R^3$ such that
for any $\varepsilon>0,$ there exists an $R>0$ satisfying
$$\liminf_{n\rightarrow+\infty}\ds\int_{B_R(y_n)}\rho_n\geq
c-\varepsilon;$$

\noindent Dichotomy:~~ there exists an $\alpha\in (0,c)$ and
$\{y_n\}\subset\R^3$ such that for all $\varepsilon>0$, $\exists$
$R>0$ satisfying
$$\limsup_{n\rightarrow+\infty}\left|\alpha-\ds\int_{B_{R}(y_n)}\rho_n\right|+\left|(c-\alpha)-\ds\int_{\R^3\backslash B_{2R}(y_n)}\rho_n\right|<\varepsilon.$$
Now we claim that compactness holds for the sequence $\{\rho_n\}$
defined in \eqref{2.11}.

\noindent(i)~~~Vanishing does not occur.

Suppose by contradiction, for all
$R>0$,$$\lim\limits_{n\rightarrow+\infty}\sup\limits_{y\in
\R^3}\ds\int_{B_R(y)}\rho_n=0,$$
then$$\lim\limits_{n\rightarrow+\infty}\sup\limits_{y\in
\R^3}\ds\int_{B_R(y)}u_n^2=0.$$By Lemma 2.10, we have that
$u_n\rightarrow0$ in $L^s(\R^3)$ for $2<s<6$. Hence by
$\{u_n\}\subset\mathcal{M}$ and $2<p<5$, we have that
$$
\begin{array}{ll}
0<c\leq
I(u_n)&=\frac{1}{2}a\ds\int_{\R^3}|Du_n|^2+\frac{1}{2}\ds\int_{\R^3}|u_n|^2+\frac{1}{4}b\left(\ds\int_{\R^3}|Du_n|^2\right)^2-\frac{1}{p+1}\ds\int_{\R^3}|u_n|^{p+1}\\[5mm]
&\leq\frac{3}{2}a\ds\int_{\R^3}|Du_n|^2+\frac{1}{2}\ds\int_{\R^3}|u_n|^2+\frac{1}{4}b\left(\ds\int_{\R^3}|Du_n|^2\right)^2-\frac{1}{p+1}\ds\int_{\R^3}|u_n|^{p+1}\\[5mm]
&=-2\ds\int_{\R^3}|u_n|^2-\frac{5}{4}b\left(\ds\int_{\R^3}|Du_n|^2\right)^2+\frac{p+3}{p+1}\ds\int_{\R^3}|u_n|^{p+1}\\[5mm]
&<\frac{p+3}{p+1}\ds\int_{\R^3}|u_n|^{p+1}\rightarrow0,
\end{array}
$$ which is impossible.

\noindent(ii)~~~Dichotomy does not occur.

Suppose by contradiction that there exists an $\alpha\in (0,c)$ and
$\{y_n\}\subset\R^3$ such that for all $\varepsilon_n\rightarrow0$,
$\exists$ $\{R_n\}\subset\R_+$ with $R_n\rightarrow+\infty$
satisfying
\begin{equation}\label{2.12}
\limsup_{n\rightarrow+\infty}\left|\alpha-\ds\int_{B_{R_n}(y_n)}\rho_n\right|+\left|(c-\alpha)-\ds\int_{\R^3\backslash
B_{2R_n}(y_n)}\rho_n\right|<\varepsilon_n.
\end{equation}
 Let $\xi:\R_+\rightarrow\R_+$ be a cut-off
function such that $0\leq\xi\leq1$, $\xi(s)\equiv1$ for $s\leq1$,
$\xi\equiv0$ for $s\geq2$ and $|\xi^\prime(s)|\leq 2.$ Set
$$v_n(x):=\xi\left(\frac{|x-y_n|}{R_n}\right)u_n(x),~~~~~~w_n(x)=\left(1-\xi\left(\frac{|x-y_n|}{R_n}\right)\right)u_n(x).$$
Then by \eqref{2.12}, we see that
$\liminf\limits_{n\rightarrow+\infty}\Phi(v_n)\geq\alpha.$
Similarly, $\liminf\limits_{n\rightarrow+\infty}\Phi(w_n)\geq
c-\alpha.$ Denote $\Omega_n:=B_{2R_n}(y_n)\backslash B_{R_n}(y_n)$,
then
$$\ds\int_{\Omega_n}\frac{a}{4}|Du_n|^2+\frac{1}{12}u_n^2+\frac{p-2}{6(p+1)}|u_n|^{p+1}=\ds\int_{\Omega_n}\rho_n\rightarrow0$$
as $n\rightarrow+\infty.$ Therefore,
$$\ds\int_{\Omega_n}a|Du_n|^2+u_n^2\rightarrow0~~~~~\hbox{and}~~~~
\ds\int_{\Omega_n}|u_n|^{p+1}\rightarrow0$$ as
$n\rightarrow+\infty$. By direct computations, we have that
$$\ds\int_{\Omega_n}a|Dv_n|^2+v_n^2\rightarrow0~~~~~\hbox{and}~~~~
\ds\int_{\Omega_n}a|Dw_n|^2+w_n^2\rightarrow0$$ as
$n\rightarrow+\infty$. Hence, we conclude that
\begin{equation}\label{2.13}
a\ds\int_{\R^3}|Du_n|^2=a\ds\int_{\R^3}|Dv_n|^2+a\ds\int_{\R^3}|Dw_n|^2+o_n(1),~~\ds\int_{\R^3}u_n^2=\ds\int_{\R^3}v_n^2+\ds\int_{\R^3}w_n^2+o_n(1),
\end{equation}
\begin{equation}\label{2.14}
\ds\int_{\R^3}|u_n|^{p+1}=\ds\int_{\R^3}|v_n|^{p+1}+\ds\int_{\R^3}|w_n|^{p+1}+o_n(1),
\end{equation}
where $o_n(1)\rightarrow0$ as $n\rightarrow+\infty$. Moreover,
\begin{equation}\label{2.15}
\begin{array}{ll}
\left(\ds\int_{\R^3}|Du_n|^2\right)^2&=\left(\ds\int_{\R^3}|Dv_n|^2+\ds\int_{\R^3}|Dw_n|^2+o_n(1)\right)^2\\[5mm]
&=\left(\ds\int_{\R^3}|Dv_n|^2\right)^2+\left(\ds\int_{\R^3}|Dw_n|^2\right)^2+2\ds\int_{\R^3}|Dv_n|^2\ds\int_{\R^3}|Dw_n|^2+o_n(1)\\[5mm]
&\geq\left(\ds\int_{\R^3}|Dv_n|^2\right)^2+\left(\ds\int_{\R^3}|Dw_n|^2\right)^2+o_n(1).
\end{array}
\end{equation}
Hence, by \eqref{2.13}\eqref{2.14}, we see that
$$\Phi(u_n)=\Phi(v_n)+\Phi(w_n)+o_n(1).$$
Then
$$c=\lim\limits_{n\rightarrow+\infty}\Phi(u_n)\geq\liminf\limits_{n\rightarrow+\infty}\Phi(v_n)+\liminf\limits_{n\rightarrow+\infty}\Phi(w_n)\geq
\alpha+c-\alpha=c,$$ hence
\begin{equation}\label{2.16}
\lim\limits_{n\rightarrow+\infty}\Phi(v_n)=\alpha,~~~~\lim\limits_{n\rightarrow+\infty}\Phi(w_n)=c-\alpha.
\end{equation}
Since $u_n\in \mathcal{M}$, $G(u_n)=0$. By
\eqref{2.13}-\eqref{2.15}, we have that
\begin{equation}\label{2.17}
0=G(u_n)\geq G(v_n)+G(w_n)+o_n(1).
\end{equation}
We have to discuss the following two cases:

\noindent Case 1.~~ Up to a subsequence, we may assume that
$G(v_n)\leq0$ or $G(w_n)\leq0$.

Without loss of generality, we suppose that $G(v_n)\leq0$, then
\begin{equation}\label{2.18}
\frac{3}{2}a\ds\int_{\R^3}|Dv_n|^2+\frac{5}{2}\ds\int_{\R^3}|v_n|^2+
\frac{3}{2}b\left(\ds\int_{\R^3}|Dv_n|^2\right)^2-\frac{p+4}{p+1}\ds\int_{\R^3}|v_n|^{p+1}\leq0.
\end{equation}
By Lemma 2.5, for any $n$, there exists $t_n>0$ such that $
(v_n)_{t_n}\in \mathcal{M}$ and then $G((v_n)_{t_n})=0$, i.e.
\begin{equation}\label{2.19}
\frac{3}{2}at_n^3\ds\int_{\R^3}|Dv_n|^2+\frac{5}{2}t_n^5\ds\int_{\R^3}|v_n|^2+
\frac{3}{2}bt_n^6\left(\ds\int_{\R^3}|Dv_n|^2\right)^2-\frac{p+4}{p+1}t_n^{p+4}\ds\int_{\R^3}|v_n|^{p+1}=0.
\end{equation}
By \eqref{2.18} and \eqref{2.19}, we have that
$$
\frac{3}{2}a(t_n^{p+1}-1)\ds\int_{\R^3}|Dv_n|^2+\frac{5}{2}(t_n^{p+1}-t_n^2)\ds\int_{\R^3}|v_n|^2+
\frac{3}{2}b(t_n^{p+1}-t_n^3)\left(\ds\int_{\R^3}|Dv_n|^2\right)^2\leq0,
$$
which implies that $t_n\leq1$. Then
\begin{equation}\label{2.20}
c\leq I((v_n)_{t_n})=\Phi((v_n)_{t_n})\leq
\Phi(v_n)\rightarrow\alpha<c,
\end{equation} which is a contradiction.

\noindent Case 2.~~ Up to a subsequence, we may assume that
$G(v_n)>0$ and $G(w_n)>0$.

By \eqref{2.17}, we see that $G(v_n)\rightarrow0$ and
$G(w_n)\rightarrow0$ as $n\rightarrow+\infty$. For $t_n$ given in
Case 1, if $\limsup\limits_{n\rightarrow+\infty}t_n\leq1$, then we
can get the same contradiction as \eqref{2.20}. Suppose now that
$\lim\limits_{n\rightarrow+\infty}t_n=t_0>1$, by \eqref{2.19}, we
have that
$$
\begin{array}{ll}
G(v_n)&=\frac{3}{2}a\ds\int_{\R^3}|Dv_n|^2+\frac{5}{2}\ds\int_{\R^3}|v_n|^2+
\frac{3}{2}b\left(\ds\int_{\R^3}|Dv_n|^2\right)^2-\frac{p+4}{p+1}\ds\int_{\R^3}|v_n|^{p+1}\\[5mm]
&=\frac{3}{2}a\left(1-\frac{1}{t_n^{p+1}}\right)\ds\int_{\R^3}|Dv_n|^2+\frac{5}{2}\left(1-\frac{1}{t_n^{p-1}}\right)\ds\int_{\R^3}|v_n|^2+
\frac{3}{2}b\left(1-\frac{1}{t_n^{p-2}}\right)\left(\ds\int_{\R^3}|Dv_n|^2\right)^2.
\end{array}
$$
Then $v_n\rightarrow0$ in $H^1(\R^3)$ since $G(v_n)\rightarrow0$,
which contradicts to \eqref{2.16} since $\alpha>0$. So dichotomy
does not occur.

Therefore, compactness holds for the sequence $\{\rho_n\},$ i.e.
there exists $\{y_n\}\subset\R^3$ such that for any $\varepsilon>0,$
there exists an $R>0$ satisfying
$$\liminf_{n\rightarrow+\infty}\ds\int_{B_R(y_n)}\frac{a}{4}|Du_n|^2+\frac{1}{12}u_n^2+\frac{p-2}{6(p+1)}|u_n|^{p+1}\geq
c-\varepsilon.$$ Hence we deduce from
$\lim\limits_{n\rightarrow+\infty}\Phi(u_n)=c$ that
$\ds\int_{\R^3\backslash B_R(y_n)}a|Du_n|^2+u_n^2\leq \varepsilon. $
\end{proof}

\begin{lem}\label{lem2.12}( \cite{wi}, Lemma 1.32)\ \
Let $\Omega$ be an open subset of $\R^N$  and let $\{u_n\}\subset L^p(\Omega),$ $1\leq p<\infty$. If $\{u_n\}$ is bounded in $L^p(\Omega)$ and $u_n\rightarrow u$ a.e. on $\Omega$, then
$$\lim\limits_{n\rightarrow\infty}(|u_n|_p^p-|u_n-u|_p^p)=|u|_p^p.$$
\end{lem}

\section{Proof of main results}
In this section, we prove our main results Theorem 1.4 and Theorem
1.1.\\

\noindent $\textbf{Proof of Theorem 1.4}$\,\,\

\begin{proof}~~
As described in $\S$ 1, to obtain Theorem 1.4, we need to prove that
problem \eqref{1.16} has a positive ground state solution in
$H^1(\R^3)$ for all $2<p<5$.

Let $\{u_n\}\subset \mathcal{M}$ be a minimizing sequence for $c$,
which was given in Lemma 2.8, then by Lemma 2.11, there exists
$\{y_n\}\subset\R^3$ such that for any $\varepsilon>0,$ there exists
an $R>0$ satisfying
\begin{equation}\label{3.1}
\ds\int_{\R^3\backslash B_R(y_n)}a|Du_n|^2+u_n^2\leq \varepsilon.
\end{equation}

Define $\tilde{u}_n(\cdot)=u_n(\cdot-y_n)\in H^1(\R^3)$, then
$\tilde{u}_n\in \mathcal{M}$. By \eqref{3.1}, we see that for any
$\varepsilon>0,$ there exists an $R>0$ such that
\begin{equation}\label{3.2}
\ds\int_{\R^3\backslash B_R(0)}a|D\tilde{u}_n|^2+\tilde{u}_n^2\leq
\varepsilon.\end{equation} Since $\{\tilde{u}_n\}$ is bounded in
$H^1(\R^3)$, up to a subsequence, we may assume that there exists a
$\tilde{u}\in H^1(\R^3)$ such that
\begin{equation}\label{3.3}
\left\{%
\begin{array}{ll}
\tilde{u}_n\rightharpoonup \tilde{u},~~~~& \hbox{$ \,\ \hbox{in}~~H^1(\R^3)$},\\
\tilde{u}_n\rightarrow \tilde{u},~~~~& \hbox{$ \,\
\hbox{in}~~L^s_{loc}(\R^3),~~\forall~s\in[2,6)$},\\
\tilde{u}_n(x)\rightarrow \tilde{u}(x),~~~~& \hbox{$ \,\
a.e.~~\hbox{in}~~\R^3$}.
\end{array}%
\right.
\end{equation}
Then by Fatou's Lemma and \eqref{3.2}, we have that
\begin{equation}\label{3.4}
\ds\int_{\R^3\backslash B_R(0)}a|D\tilde{u}|^2+\tilde{u}^2\leq
\varepsilon.\end{equation} By \eqref{3.2}-\eqref{3.4} and the
Sobolev embedding theorem, we see that for any $s\in[2,6)$ and any
$\varepsilon>0,$ there exists a $C>0$ such that
\begin{equation}\label{3.5}
\begin{array}{ll}
\ds\int_{\R^3}|\tilde{u}_n-\tilde{u}|^s&\leq\ds\int_{B_R(0)}|\tilde{u}_n-\tilde{u}|^s+\ds\int_{\R^3\backslash
B_R(0)}|\tilde{u}_n-\tilde{u}|^s\\[5mm]
&\leq\varepsilon+C(\|\tilde{u}_n\|_{H^1(\R^3\backslash
B_R(0))}+\|\tilde{u}\|_{H^1(\R^3\backslash
B_R(0))})\\[5mm]
&\leq (1+2C)\varepsilon.
\end{array}
\end{equation}
Then \begin{equation}\label{3.6} \tilde{u}_n\rightarrow
\tilde{u}~~~~\hbox{in}~~L^s(\R^3)~~\hbox{for~any}~s\in[2,6).
\end{equation}
Since $\tilde{u}_n\in \mathcal{M}$, by Lemma 2.7,
$|\tilde{u}_n|_{p+1}\geq C$ for some $C>0$, hence
$|\tilde{u}|_{p+1}\geq C>0$, which implies that $\tilde{u}\neq0$.

We next show that $\tilde{u}_n\rightarrow\tilde{u}$ in $H^1(\R^3)$.
Indeed, by \eqref{3.3} \eqref{3.6} and Fatou's Lemma, we have that
$$\alpha\triangleq a\ds\int_{\R^3}|D\tilde{u}|^2\leq\liminf\limits_{n\rightarrow+\infty}a\ds\int_{\R^3}|D\tilde{u}_n|^2\triangleq\tilde{\alpha},$$
$$\beta\triangleq\ds\int_{\R^3}|\tilde{u}|^2\leq\liminf\limits_{n\rightarrow+\infty}\ds\int_{\R^3}|\tilde{u}_n|^2\triangleq\tilde{\beta},$$
$$ b\left(\ds\int_{\R^3}|D\tilde{u}|^2\right)^2\leq\liminf\limits_{n\rightarrow+\infty}b\left(\ds\int_{\R^3}|D\tilde{u}_n|^2\right)^2$$
and
$$\ds\int_{\R^3}|\tilde{u}|^{p+1}=\lim\limits_{n\rightarrow+\infty}\ds\int_{\R^3}|\tilde{u}_n|^{p+1}.$$
Then
\begin{equation}\label{3.7}
G(\tilde{u})\leq\liminf\limits_{n\rightarrow+\infty}G(\tilde{u}_n)=0.
\end{equation}
Just suppose that $\alpha+\beta<\tilde{\alpha}+\tilde{\beta}$, then
$I(\tilde{u})<c$ and $G(\tilde{u})<0$, therefore, $\tilde{u}\notin
\mathcal{M}$. By Lemma 2.5, there exists a $0<t_0<1$ such that
$\tilde{u}_{t_0}\in \mathcal{M}$. Since $G(\tilde{u}_{t_0})=0$ and
$G(\tilde{u})<0$, $t_0<1$, then we see that
$$I(\tilde{u}_{t_0})=\Phi(\tilde{u}_{t_0})<\Phi(\tilde{u})\leq\lim\limits_{n\rightarrow+\infty}\Phi(u_n)=\lim\limits_{n\rightarrow+\infty}I(u_n)=c,$$
which is impossible, where $\Phi$ is given in \eqref{2.9}. Then
$\alpha+\beta=\tilde{\alpha}+\tilde{\beta}$. So
$\tilde{u}_n\rightarrow\tilde{u}$ in $H^1(\R^3)$.

We deduce that $\tilde{u}\in \mathcal{M}$ and $I(\tilde{u})=c$, i.e.
$I|_{\mathcal{M}}$ attains its minimum at $\tilde{u}$, then
$\tilde{u}$ is a nontrivial critical point of $I|_{\mathcal{M}}$,
hence by Lemma 2.6, we see that $\tilde{u}$ is a ground state
solution of \eqref{1.16}.

It is easy to see that $|\tilde{u}|$ is also a ground state solution
of \eqref{1.16} since the functional $I$ and the manifold
$\mathcal{M}$ are symmetric, hence we may assume that such a ground
state solution does not change sign, i.e. $\tilde{u}\geq0$. By using
the strong maximum principle and standard arguments, see e.g.
\cite{af,be,li,m,to,tr}, we obtain that $\tilde{u}(x)>0$ for all
$x\in\R^3$. Therefore, $\tilde{u}$ is a positive ground state
solution of \eqref{1.16} and the proof is completed.
\end{proof}

Assume that $(V_1)-(V_3)$ hold, we apply Proposition 1.3 to prove
Theorem 1.1.

Set $T=[\delta,1]$, where $\delta\in(0,1)$ is a positive constant.
We consider a family of functionals on $H^1(\R^3)$
\begin{equation}\label{4.1}
I_{V,\lambda}(u)=\frac{1}{2}\ds\int_{\R^3}(a|Du|^2+V(x)|u|^2)+\frac{b}{4}\left(\ds\int_{\R^3}|Du|^2\right)^2-\frac{\lambda}{p+1}\ds\int_{\R^3}|u|^{p+1},~~\forall~\lambda\in
[\delta,1].
\end{equation}
Then $I_{V,\lambda}(u)=A(u)-\lambda B(u),$
where$$A(u)=\frac{1}{2}\ds\int_{\R^3}(a|Du|^2+V(x)|u|^2)+\frac{b}{4}\left(\ds\int_{\R^3}|Du|^2\right)^2\rightarrow+\infty~~\hbox{as}~~\|u\|\rightarrow+\infty,$$
$$B(u)=\frac{1}{p+1}\ds\int_{\R^3}|u|^{p+1}\geq0.$$

\begin{lem}\label{lem4.1}\ \
Assume that $(V_2)(V_3)$ hold and $2<p<5$, then

(i)~~ there exists a $v\in H^1(\R^3)\backslash\{0\}$ such that
$I_{V,\lambda}(v)\leq0$ for all $\lambda\in [\delta,1]$;

(ii)~~$c_\lambda=\inf\limits_{\gamma\in
\Gamma}\max\limits_{t\in[0,1]}I_{V,\lambda}(\gamma(t))>\max\{I_{V,\lambda}(0),~I_{V,\lambda}(v)\}$
for all $\lambda\in [\delta,1]$, where $\Gamma=\{\gamma\in
C([0,1],H^1(\R^3))|~\gamma(0)=0,~\gamma(1)=v\}.$
\end{lem}
\begin{proof}~~

\noindent(i)~~For fixed $u\in H^1(\R^3)\backslash\{0\}$ and any
$\lambda\in [\delta,1]$, we have that
$$I_{V,\lambda}(u)\leq I^\infty_\delta(u)=\frac{1}{2}\ds\int_{\R^3}(a|Du|^2+V_\infty|u|^2)+\frac{b}{4}\left(\ds\int_{\R^3}|Du|^2\right)^2-\frac{\delta}{p+1}\ds\int_{\R^3}|u|^{p+1}.$$
Set $u_t(x)=tu(t^{-1}x)$, $\forall~t>0$, by Lemma 2.2, then
$I^\infty_\delta(u_t)\rightarrow-\infty$ as $t\rightarrow+\infty$.
Hence, take $v=u_t$ for $t$ large, we have that
$I_{V,\lambda}(v)\leq I^\infty_\delta(v)<0$.

\noindent(ii)~~Since
$$I_{V,\lambda}(u)\geq
\frac{1}{2}\|u\|^2-\frac{\lambda}{p+1}\ds\int_{\R^3}|u|^{p+1}\geq\frac{1}{2}\|u\|^2-\frac{C}{p+1}\|u\|^{p+1}$$
and $p>2$, we see that $I_{V,\lambda}$ has a strict local minimum at
0 and hence $c_\lambda>0$.
\end{proof}

Lemma 3.1 and the definition of $I_{V,\lambda}(u)$ imply that
$I_{V,\lambda}(u)$ satisfies the assumptions of Proposition 1.3 with
$X=H^1(\R^3)$ and $\Phi_\lambda=I_{V,\lambda}$. So for a.e.
$\lambda\in [\delta,1]$, there exists a bounded sequence
$\{u_n\}\subset H^1(\R^3)$ (for simplicity, we denote $\{u_n\}$
instead of $\{u_n(\lambda)\}$) such that
$$I_{V,\lambda}(u_n)\rightarrow
c_\lambda,~~~~I^{\prime}_{V,\lambda}(u_n)\rightarrow0~~\hbox{in}~H^1(\R^3).$$

\begin{lem}\label{lem4.2}(\cite{j}, Lemma 2.3)\ \
Under the assumptions of Proposition 1.3, the map
$\lambda\rightarrow c_\lambda$ is non-increasing and left
continuous.
\end{lem}

By Theorem 1.4, we see that for any $\lambda\in [\delta,1]$, the
associated limit problem
\begin{equation}\label{4.2}
\left\{%
\begin{array}{ll}
    -\left(a+b\ds\int_{\R^3}|D u|^2\right)\Delta u+V_\infty u=\lambda|u|^{p-1}u, & \hbox{$x\in \R^3$}, \\
    u\in H^1(\R^3),~~~~u>0, & \hbox{$x\in \R^3$},\\
\end{array}%
\right.\end{equation} where $2<p<5$, has a positive ground state
solution in $ H^1(\R^3)$, i.e. for any $\lambda\in [\delta,1]$,
\begin{equation}\label{4.3}
m_{\lambda}^\infty\triangleq\inf_{u\in
\mathcal{M}^\infty_\lambda}I^\infty_{\lambda}(u)
\end{equation}
is achieved at some
$u_\lambda^\infty\in\mathcal{M}^\infty_\lambda\triangleq\{u\in
H^1(\R^3)\backslash\{0\}|~G^\infty_\lambda(u)=0\}$ and
$I^{\prime\infty}_\lambda(u_\lambda^\infty)=0$, where
\begin{equation}\label{4.4}
I^\infty_{\lambda}(u)=\frac{1}{2}\ds\int_{\R^3}(a|Du|^2+V_\infty|u|^2)+\frac{b}{4}\left(\ds\int_{\R^3}|Du|^2\right)^2-\frac{\lambda}{p+1}\ds\int_{\R^3}|u|^{p+1}
\end{equation}
and
$$G_\lambda^\infty(u)\triangleq\frac{3}{2}a\ds\int_{\R^3}|Du|^2+\frac{5}{2}\ds\int_{\R^3}V_\infty|u|^2+\frac{3}{2}b\left(\ds\int_{\R^3}|Du|^2\right)^2-\frac{(p+4)\lambda}{p+1}\ds\int_{\R^3}|u|^{p+1}.
$$
\begin{lem}\label{lem4.3}\ \
Assume that $(V_1)-(V_3)$ hold and $2<p<5$, then
$c_\lambda<m_\lambda^\infty$ for any $\lambda\in [\delta,1]$.
\end{lem}
\begin{proof}~~
 Let $u^\infty_\lambda$ be the minimizer of
$m^\infty_\lambda$, by Lemma 2.5, we have that
$I^\infty_\lambda(u^\infty_\lambda)=\max\limits_{t>0}I^\infty_\lambda(tu^\infty_\lambda(t^{-1}x)).$
Then choosing $v(x)=tu^\infty_\lambda(t^{-1}x)$ for $t$ large in
Lemma 3.1 (i), by $(V_2)$, we see that for $\forall$
$\lambda\in[\delta,1]$,
$$c_\lambda\leq\max_{t>0}I_{V,\lambda}(tu^\infty_\lambda(t^{-1}x))<\max_{t>0}I^\infty_\lambda(tu^\infty_\lambda(t^{-1}x))=I^\infty_\lambda(u^\infty_\lambda)=m^\infty_\lambda.$$
\end{proof}

In order to prove that the functional $I_{V,\lambda}$ satisfies
$(PS)_{c_\lambda}$ condition for a.e. $\lambda\in [\delta,1]$, we
need the following new version of a global compactness lemma, which is suitable for Kirchhoff equations.

\begin{lem}\label{lem4.4}\ \
Assume that $(V_2)(V_3)$ hold and $2<p<5$. For $c>0$ and
$\forall~\lambda\in[\delta,1]$, let $\{u_n\}\subset H^1(\R^3)$ be a
bounded $(PS)_{c}$ sequence for $I_{V,\lambda}$, then there exists a
$u\in H^1(\R^3)$ and $A\in\R$ such that $J_{V,\lambda}^\prime(u)=0$, where
\begin{equation}\label{4.13}
J_{V,\lambda}(u)=\frac{a+bA^2}{2}\ds\int_{\R^3}|Du|^2+\frac{1}{2}\ds\int_{\R^3}V(x)|u|^2-\frac{\lambda}{p+1}\ds\int_{\R^3}|u|^{p+1}
\end{equation}
and either

(i)~~$u_n\rightarrow u$ in $H^1(\R^3)$,

\noindent or

(ii)~~there exists an $l\in \mathbb{N}$ and $\{y_n^k\}\subset\R^3$
with $|y_n^k|\rightarrow\infty$ as $n\rightarrow\infty$ for each $1\leq k\leq l$, nontrivial
solutions $w^1,\cdots,w^l$ of the following problem
\begin{equation}\label{4.5}
-(a+bA^2)\Delta u+V_\infty u=\lambda|u|^{p-1}u
\end{equation}
such that
$$
c+\frac{bA^4}{4}=
J_{V,\lambda}(u)+\sum_{k=1}^{l}J_\lambda^\infty(w^k)
$$
 where \begin{equation}\label{4.14}
 J^\infty_{\lambda}(u)=\frac{a+bA^2}{2}\ds\int_{\R^3}|Du|^2+\frac{1}{2}\ds\int_{\R^3}V_\infty|u|^2-\frac{\lambda}{p+1}\ds\int_{\R^3}|u|^{p+1}\end{equation}
  and
$$\left\|u_n-u-\sum_{k=1}^{l}w^k(\cdot-y_n^k)\right\|\rightarrow0,$$
$$A^2=|Du|_2^2+\sum_{k=1}^{l}|Dw^k|_2^2.$$
\end{lem}
\begin{proof}~~
Since $\{u_n\}$ is bounded in $H^1(\R^3)$, there exists a $u\in
H^1(\R^3)$ and $A\in\R$ such that
\begin{equation}\label{4.6}
u_n\rightharpoonup u~~\hbox{in}~H^1(\R^3)
\end{equation}
and \begin{equation}\label{4.7}
\ds\int_{\R^3}|Du_n|^2\rightarrow A^2.
\end{equation}
Then $I^\prime_{V,\lambda}(u_n)\rightarrow0$ implies that
$$\ds\int_{\R^3}(aDuD\varphi+V(x)u\varphi)+bA^2\ds\int_{\R^3}DuD\varphi-\lambda\ds\int_{\R^3}|u|^{p-1}u\varphi=0,~~\forall~\varphi\in H^1(\R^3),$$
i.e. $J_{V,\lambda}^\prime(u)=0$, where
$$J_{V,\lambda}(u)=\frac{a+bA^2}{2}\ds\int_{\R^3}|Du|^2+\frac{1}{2}\ds\int_{\R^3}V(x)|u|^2-\frac{\lambda}{p+1}\ds\int_{\R^3}|u|^{p+1}.$$

Since
$$\begin{array}{ll}
J_{V,\lambda}(u_n)&=\frac{a+bA^2}{2}\ds\int_{\R^3}|Du_n|^2+\frac{1}{2}\ds\int_{\R^3}V(x)|u_n|^2-\frac{\lambda}{p+1}\ds\int_{\R^3}|u_n|^{p+1}\\[5mm]
&=\frac{a}{2}\ds\int_{\R^3}|Du_n|^2+\frac{1}{2}\ds\int_{\R^3}V(x)|u_n|^2+\frac{b}{4}\left(\ds\int_{\R^3}|Du_n|^2\right)^2-\frac{\lambda}{p+1}\ds\int_{\R^3}|u_n|^{p+1}\\[5mm]
&~~~~~+\frac{bA^2}{4}\ds\int_{\R^3}|Du_n|^2+o(1)\\[5mm]
&=I_{V,\lambda}(u_n)+\frac{bA^4}{4}+o(1)
\end{array}
$$
and$$\begin{array}{ll}
\langle J^\prime_{V,\lambda}(u_n),\varphi\rangle&=(a+bA^2)\ds\int_{\R^3}Du_nD\varphi+\ds\int_{\R^3}V(x)u\varphi-\lambda\ds\int_{\R^3}|u|^{p-1}u\varphi\\[5mm]
&=a\ds\int_{\R^3}Du_nD\varphi+\ds\int_{\R^3}V(x)u\varphi+b\ds\int_{\R^3}|Du_n|^2\ds\int_{\R^3}Du_nD\varphi-\lambda\ds\int_{\R^3}|u|^{p-1}u\varphi+o(1)\\[5mm]
&=\langle I^\prime_{V,\lambda}(u_n),\varphi\rangle+o(1),
\end{array}
$$
we conclude that
$$J_{V,\lambda}(u_n)\rightarrow c+\frac{bA^4}{4},~~~~~~~~~J^\prime_{V,\lambda}(u_n)\rightarrow0~~\hbox{in}~H^{-1}(\R^3).$$
We next show that either (i) or (ii) holds. The argument is similar to \cite{fmm}, for reader's convenience, we give a detailed proof.\\

\noindent Step 1: Set $u_n^1=u_n-u$, by \eqref{4.6}, Lemma 2.12 and $(V_2)$ we see that

$(a.1)$~~$|Du^1_n|_2^2=|Du_n|_2^2-|Du|_2^2+o(1)$,

$(b.1)$~~$|u^1_n|_2^2=|u_n|_2^2-|u|_2^2+o(1)$,

$(c.1)$~~$J^\infty_{\lambda}(u_n^1)\rightarrow c+\frac{bA^4}{4}-J_{V,\lambda}(u),$

$(d.1)$~~$(J^\infty_{\lambda})^\prime(u_n^1)\rightarrow0$ in $H^{-1}(\R^3)$.

Let $$\sigma^1=\limsup\limits_{n\rightarrow\infty}\sup\limits_{y\in \R^3}\ds\int_{B_1(y)}|u_n^1|^2.$$
\noindent \textbf{Vanishing}:~~ If $\sigma^1=0,$ then it follows from Lemma 2.10 that $u_n^1\rightarrow0$ in $L^s(\R^3)$ for $\forall~s\in(2,2^*)$. Since $(J^\infty_{\lambda})^\prime(u_n^1)\rightarrow0$, we see that $u_n^1\rightarrow 0$ in $H^1(\R^3)$ and the proof is completed.

\noindent \textbf{Non-vanishing}:~~ If $\sigma^1>0$, then there exists a sequence $\{y_n^1\}\subset\R^3$ such that
 $$\ds\int_{B_1(y_n^1)}|u_n^1|^2>\frac{\sigma^1}{2}.$$
 Set $w_n^1\triangleq u_n^1(\cdot+y_n^1)$. Then $\{w_n^1\}$ is bounded in $H^1(\R^3)$ and we may assume that $w_n^1\rightharpoonup w^1$ in $H^1(\R^3)$. Hence $(J^\infty_{\lambda})^\prime(w^1)=0$. Since $$\ds\int_{B_1(0)}|w_n^1|^2>\frac{\sigma^1}{2},$$
 we see that $w^1\neq0$. Moreover, $u_n^1\rightharpoonup 0$ in $H^1(\R^3)$ implies that $\{y_n^1\}$ is unbounded. Hence, we may assume that $|y_n^1|\rightarrow\infty$.\\

\noindent Step 2: Set $u_n^2=u_n-u-w^1(\cdot-y_n^1)$. We can similarly check that

 $(a.2)$~~$|Du^2_n|_2^2=|Du_n|_2^2-|Du|_2^2-|Dw^1|_2^2+o(1)$,

$(b.2)$~~$|u^2_n|_2^2=|u_n|_2^2-|u|_2^2-|w^1|_2^2+o(1)$,

$(c.2)$~~$J^\infty_{\lambda}(u_n^2)\rightarrow c+\frac{bA^4}{4}-J_{V,\lambda}(u)-J^\infty_{\lambda}(w^1),$

$(d.2)$~~$(J^\infty_{\lambda})^\prime(u_n^2)\rightarrow0$ in $H^{-1}(\R^3)$.

Similar to the arguments in Step 1, let
$$\sigma^2=\limsup\limits_{n\rightarrow\infty}\sup\limits_{y\in \R^3}\ds\int_{B_1(y)}|u_n^2|^2.$$
If vanishing occurs, then $\|u_n^2\|\rightarrow0$, i.e. $\|u_n-u-w^1(\cdot-y_n^1)\|\rightarrow0$. Moreover, by \eqref{4.7} and $(a.2)$ $(c.2)$, we see that $$A^2=|Du|_2^2+|Dw^1|^2_2~~\hbox{and}~~c+\frac{bA^4}{4}=J_{V,\lambda}(u)+J^\infty_{\lambda}(w^1).$$

If non-vanishing occurs, then there exists a sequence $\{y_n^2\}\subset\R^3$ and a nontrivial $w^2\in H^1(\R^3)$ such that $w_n^2\triangleq u_n^2(\cdot+y_n^2)\rightharpoonup w^2$ in $H^1(\R^3)$. Then by $(d.2)$, we have that $(J^\infty_{\lambda})^\prime(w^2)=0$. Furthermore, $u_n^2\rightharpoonup0$ in $H^1(\R^3)$ implies that $|y_n^2|\rightarrow\infty$ and $|y_n^2-y_n^1|\rightarrow\infty$.

We next proceed by iteration. Recall that if $w^k$ is a nontrivial solution of $I^\infty_{\lambda}$, then $I^\infty_{\lambda}(w^k)>0$. So there exists some finite $l\in\mathbb{N}$ such that only the vanishing case occurs in Step $l$. Then the lemma is proved.
\end{proof}

\begin{lem}\label{lem4.5}\ \
Assume that $(V_1)-(V_3)$ hold and $2<p<5$. For
$\lambda\in[\delta,1]$, let $\{u_n\}\subset H^1(\R^3)$ be a bounded
$(PS)_{c_\lambda}$ sequence of $I_{V,\lambda}$, then there exists a nontrivial $u_\lambda\in H^1(\R^3)$ such that $$u_n\rightarrow
u_\lambda~~~\hbox{in}~H^1(\R^3).$$
\end{lem}
\begin{proof}~~
By Lemma 3.4, for $\lambda\in[\delta,1],$ there exists a
$u_\lambda\in H^1(\R^3)$ and $A_\lambda\in\R$ such that
$$u_n\rightharpoonup u_\lambda~~~\hbox{in}~H^1(\R^3),$$
$$\ds\int_{\R^3}|Du_n|^2\rightarrow A_\lambda^2$$
and $J^\prime_{V,\lambda}(u_\lambda)=0$ and either (i) or (ii)
occurs, where $J_{V,\lambda}$ is given in \eqref{4.13}.

If (ii) occurs, i.e. there exists an $l\in \mathbb{N}$ and $\{y_n^k\}\subset\R^3$
with $|y_n^k|\rightarrow\infty$ as $n\rightarrow\infty$ for each $1\leq k\leq l$, nontrivial
solutions $w^1,\cdots,w^l$ of problem \eqref{4.5}
such that
$$c+\frac{bA_\lambda^4}{4}=
J_{V,\lambda}(u_\lambda)+\sum_{k=1}^{l}J_\lambda^\infty(w^k)$$ and
$$\left\|u_n-u_\lambda-\sum_{k=1}^{l}w^k(\cdot-y_n^k)\right\|\rightarrow0,$$
\begin{equation}\label{4.8}
A_\lambda^2=|Du_\lambda|_2^2+\sum_{k=1}^{l}|Dw^k|_2^2,
\end{equation}
where $J^\infty_\lambda$ is given in \eqref{4.14}.

Denote
$$
\left\{%
\begin{array}{ll}\alpha\triangleq a\ds\int_{\R^3}|Du_\lambda|^2,~~~~
\beta\triangleq
\ds\int_{\R^3}V(x)|u_\lambda|^2,~~~~\bar{\beta}\triangleq
\ds\int_{\R^3}(DV(x),x)|u_\lambda|^2,\\
\mu\triangleq bA^2_\lambda\ds\int_{\R^3}|Du_\lambda|^2,
~~~~~~\theta\triangleq\ds\int_{\R^3}|u_\lambda|^{p+1}.
\end{array}%
\right.
$$
 Then $\alpha,$ $\mu,$
$\theta$ must be nonnegative and by $(V_1)$, $\beta-\bar{\beta}\geq0$.
By the Poho\u{z}aev identity and
$J^\prime_{V,\lambda}(u_\lambda)=0$, we have that
$$
\left\{%
\begin{array}{ll}
\frac{1}{2}\alpha+\frac{3}{2}\beta+\frac{1}{2}\bar{\beta}+\frac{1}{2}\mu-\frac{3\lambda}{p+1}\theta=0, \\
\frac{1}{2}\alpha+\frac{1}{2}\beta+\frac{1}{4}\mu-\frac{\lambda}{p+1}\theta=J_{V,\lambda}(u_\lambda)-\frac{1}{4}\mu, \\
\alpha+\beta+\mu-\lambda\theta=0.\\
\end{array}%
\right.
$$
Then we conclude that
$$6\left(J_{V,\lambda}(u_\lambda)-\frac{1}{4}\mu\right)=\frac{3}{2}\alpha+\frac{1}{2}(\beta-\bar{\beta})+\frac{p-2}{p+1}\lambda\theta\geq\frac{p-2}{p+1}\lambda\theta\geq0.$$
Hence
\begin{equation}\label{4.9}
J_{V,\lambda}(u_\lambda)\geq\frac{1}{4}bA^2_\lambda\ds\int_{\R^3}|Du_\lambda|^2.
\end{equation}
For each nontrivial solution $w^k$ $(k=1,\cdots,l)$ of problem \eqref{4.5}, i.e. $(J^\infty_{\lambda})^\prime(w^k)=0$. Recall that $w^k$ satisfies the Poho\u{z}ave identity
$$\widetilde{P_\lambda}(w^k)\triangleq\frac{a+bA_\lambda^2}{2}\ds\int_{\R^3}|Dw^k|^2+\frac{3}{2}\ds\int_{\R^3}V_\infty|w^k|^2-\frac{3\lambda}{p+1}\ds\int_{\R^3}|w^k|^{p+1}=0.$$
Then by \eqref{4.8}, we have that
$$\begin{array}{ll}
0&=\langle (J^\infty_{\lambda})^\prime(w^k),w^k\rangle+\widetilde{P_\lambda}(w^k)\\[5mm]
&=\frac{3(a+bA_\lambda^2)}{2}\ds\int_{\R^3}|Dw^k|^2+\frac{5}{2}\ds\int_{\R^3}V_\infty|w^k|^2-\frac{(p+4)\lambda}{p+1}\ds\int_{\R^3}|w^k|^{p+1}\\[5mm]
&\geq G^\infty_\lambda(w^k).
\end{array}
$$
Hence there exists $t_k\in(0,1]$ such that $t_kw^k(t_k^{-1}x)\in \mathcal{M}_\lambda^\infty$. So by \eqref{4.8}, we see that
\begin{equation}\label{4.10}\begin{array}{ll}
J_\lambda^\infty(w^k)&=\left[J_\lambda^\infty(w^k)-\frac{\langle (J^\infty_{\lambda})^\prime(w^k),w^k\rangle+\widetilde{P}(w^k)}{p+4}-\frac{bA_\lambda^2}{4}\ds\int_{\R^3}|Dw^k|^2\right]+\frac{bA_\lambda^2}{4}\ds\int_{\R^3}|Dw^k|^2\\[5mm]
&=\frac{(p+1)a}{2(p+4)}\ds\int_{\R^3}|Dw^k|^2+\frac{p-1}{2(p+4)}\ds\int_{\R^3}V_\infty|w^k|^2+\frac{b(p-2)}{4(p+4)}A_\lambda^2\ds\int_{\R^3}|Dw^k|^2+\frac{bA_\lambda^2}{4}\ds\int_{\R^3}|Dw^k|^2\\[5mm]
&\geq\frac{(p+1)a}{2(p+4)}\ds\int_{\R^3}|Dw^k|^2+\frac{p-1}{2(p+4)}\ds\int_{\R^3}V_\infty|w^k|^2+\frac{b(p-2)}{4(p+4)}\left(\ds\int_{\R^3}|Dw^k|^2\right)^2+\frac{bA_\lambda^2}{4}\ds\int_{\R^3}|Dw^k|^2\\[5mm]
&\geq\frac{(p+1)a}{2(p+4)} t_k^3\ds\int_{\R^3}|Dw^k|^2+\frac{p-1}{2(p+4)}t_k^5\ds\int_{\R^3}V_\infty|w^k|^2+\frac{b(p-2)}{4(p+4)}t_k^3\left(\ds\int_{\R^3}|Dw^k|^2\right)^2\\[5mm]
&~~~~~~~~+\frac{bA_\lambda^2}{4}\ds\int_{\R^3}|Dw^k|^2\\[5mm]
&=I^\infty_{\lambda}(t_kw^k(t_k^{-1}x))-\frac{1}{p+4} G^\infty_\lambda(t_kw^k(t_k^{-1}x))+\frac{bA_\lambda^2}{4}\ds\int_{\R^3}|Dw^k|^2\\[5mm]
&=I^\infty_{\lambda}(t_kw^k(t_k^{-1}x))+\frac{bA_\lambda^2}{4}\ds\int_{\R^3}|Dw^k|^2\\[5mm]
&\geq m_\lambda^\infty+\frac{bA_\lambda^2}{4}\ds\int_{\R^3}|Dw^k|^2.
\end{array}
\end{equation}

Then by \eqref{4.8}-\eqref{4.10}, we have that
$$
\begin{array}{ll}
c_\lambda+\frac{bA_\lambda^4}{4}&=J_{V,\lambda}(u_\lambda)+\sum_{k=1}^{l}J_\lambda^\infty(w^k)\\[5mm]
&\geq lm_\lambda^\infty+\frac{bA_\lambda^2}{4}\ds\int_{\R^3}|Du_\lambda|^2+\frac{bA_\lambda^2}{4}\sum_{k=1}^{l}\ds\int_{\R^3}|Dw^k|^2\\[5mm]
&\geq m_\lambda^\infty+\frac{bA_\lambda^4}{4}.
\end{array}
$$
i.e. $c_\lambda\geq m_\lambda^\infty$, which contradicts to Lemma 3.3. So (i) holds,
i.e. $u_n\rightarrow u_\lambda$ in $H^1(\R^3)$ and then $u_\lambda$
is a nontrivial critical point for $I_{V,\lambda}$ and
$I_{V,\lambda}(u_\lambda)=c_\lambda.$
\end{proof}

\noindent $\textbf{Proof of Theorem 1.1}$\,\,\

\noindent\begin{proof} \,\,\,\,\ We complete the proof in two steps.

\noindent\textbf{Step 1}:~~~By Lemma 3.1, Lemma 3.5 and Proposition
1.3, for a.e. $\lambda\in [\delta,1],$ there exists a nontrivial
critical point $u_\lambda\in H^1(\R^3)$ for $I_{V,\lambda}$ and
$I_{V,\lambda}(u_\lambda)=c_\lambda.$

Choosing a sequence $\{\lambda_n\}\subset[\delta,1]$ satisfying
$\lambda_n\rightarrow1$, then we have a sequence of nontrivial
critical points $\{u_{\lambda_n}\}$ of $I_{V,\lambda_n}$ and
$I_{V,\lambda_n}(u_{\lambda_n})=c_{\lambda_n}$. We next show that
$\{u_{\lambda_n}\}$ is bounded in $H^1(\R^3)$.

Denote $$
\left\{%
\begin{array}{ll}\alpha_n\triangleq a\ds\int_{\R^3}|Du_{\lambda_n}|^2,~~~~
\beta_n\triangleq
\ds\int_{\R^3}V(x)|u_{\lambda_n}|^2,~~~~\bar{\beta}_n\triangleq
\ds\int_{\R^3}(DV(x),x)|u_{\lambda_n}|^2,\\
\mu_n\triangleq b\left(\ds\int_{\R^3}|Du_{\lambda_n}|^2\right)^2,
~~~~~~\theta_n\triangleq\ds\int_{\R^3}|u_{\lambda_n}|^{p+1}.
\end{array}%
\right.
$$Then
$$\left\{%
\begin{array}{ll}
\frac{1}{2}\alpha_n+\frac{3}{2}\beta_n+\frac{1}{2}\bar{\beta}_n+\frac{1}{2}\mu_n-\frac{3\lambda_n}{p+1}\theta_n=0, \\
\frac{1}{2}\alpha_n+\frac{1}{2}\beta_n+\frac{1}{4}\mu_n-\frac{\lambda_n}{p+1}\theta_n=c_{\lambda_n}, \\
\alpha_n+\beta_n+\mu_n-\lambda_n\theta_n=0.\\
\end{array}%
\right.
$$
Hence
\begin{equation}\label{4.11}
\frac{3}{2}\alpha_n+\frac{1}{2}(\beta_n-\bar{\beta}_n)+\frac{p-2}{p+1}\lambda_n\theta_n=6c_{\lambda_n}\leq
6c_{\delta}\end{equation}
 and
\begin{equation}\label{4.12}
\frac{1}{4}(\alpha_n+\beta_n)+\left(\frac{1}{4}-\frac{1}{p+1}\right)\lambda_n\theta_n=c_{\lambda_n}.
\end{equation}
Since $(V_1)$ implies that $\beta_n-\bar{\beta}_n\geq0$ and $\alpha_n$, $\mu_n$, $\theta_n$ are nonnegative, we conclude
that $\theta_n$ is bounded from \eqref{4.11}, hence by \eqref{4.12},
we have that $\alpha_n+\beta_n$ is bounded, i.e. $\{u_{\lambda_n}\}$
is bounded in $H^1(\R^3)$. Therefore by Lemma 3.2, we see that
$$
\lim\limits_{n\rightarrow\infty}I_{V,1}(u_{\lambda_n})
=\lim\limits_{n\rightarrow\infty}\left(I_{V,\lambda_n}(u_{\lambda_n})+(\lambda_n-1)\ds\int_{\R^3}|u_{\lambda_n}|^{p+1}\right)
=\lim\limits_{n\rightarrow\infty}c_{\lambda_n}=c_1
$$
and
$$
\lim\limits_{n\rightarrow\infty}\langle
I^\prime_{V,1}(u_{\lambda_n}),\varphi\rangle
=\lim\limits_{n\rightarrow\infty}\left(\langle
I^\prime_{V,\lambda_n}(u_{\lambda_n}),\varphi\rangle+(\lambda_n-1)\ds\int_{\R^3}|u_{\lambda_n}|^{p-1}u_{\lambda_n}\varphi\right)=0,
$$
i.e. $\{u_{\lambda_n}\}$ is a bounded $(PS)_{c_1}$ sequence for $I_V=I_{V,1}$. Then by Lemma 3.5, there exists a nontrivial critical point
$u_0\in H^1(\R^3)$ for $I_V$ and $I_V(u_0)=c_1$.

\noindent\textbf{Step 2}:~~~Next we prove the existence of a ground
state solution for problem \eqref{1.1}. Set
$$m=\inf\{I_V(u)|~u\neq0,~I_V^\prime(u)=0\}.$$
Then by $(V_1)$, we see that $0< m\leq I_V(u_0)=c_1<+\infty$. Let
$\{u_n\}$ be a sequence of nontrivial critical points of $I_V$
satisfying $I_V(u_n)\rightarrow m$, using the same arguments as in
\textbf{Step 1}, we can deduce that $\{u_n\}$ is bounded in
$H^1(\R^3)$, i.e. $\{u_n\}$ is a bounded $(PS)_m$ sequence of $I_V$.
Similar to the arguments in Lemma 3.5, there exists a nontrivial $u\in H^1(\R^3)$ such that
$I_V(u)=m$ and $I^\prime_V(u)=0$. By the standard regularity
arguments as in the proof of Theorem 1.4, we see that $u$ is a
positive ground state solution for problem \eqref{1.1}. Then the
proof is complete.
\end{proof}

\section{Proof of Theorem 1.6}
\noindent $\textbf{Proof of Theorem 1.6}$\,\,\

\noindent \begin{proof} \,\,\,\,\ Suppose that $u\in H^1(\R^3)$ is a
nontrivial solution to \eqref{1.11}, multiplying the equation
\eqref{1.11} by $u$ and integrating, we have that
$$a\ds\int_{\R^3}
(|Du|^2+V(x)|u|^2)+b\lambda\left(\ds\int_{\R^3}(|Du|^2+V(x)|u|^2)\right)^2-\ds\int_{\R^3}|u|^{p+1}=0.
$$
Since $a>1$, for $t\geq0$, set
$$g(t)\triangleq t^4b\lambda
\left(\ds\int_{\R^3}(|Du|^2+V(x)|u|^2)\right)^2+t^2(a-1)\ds\int_{\R^3}
(|Du|^2+V(x)|u|^2)-t^3\ds\int_{\R^3}|u|^3.$$ Denote $C>0$ be the
best Sobolev constant for the embedding from $H^1(\R^3)$ into
$L^3(\R^3)$, i.e.
$C=\inf\limits_{H^1(\R^3)\backslash\{0\}}\frac{\int_{\R^3}|Du|^2+V(x)|u|^2}{|u|_3^2}.$
In particular,
\begin{equation}\label{5.1}
\left(\int_{\R^3}|u|^3\right)^{\frac{1}{3}}\leq
C^{-\frac{1}{2}}\|u\|,~~~~~\forall~u\in H^1(\R^3).
\end{equation}
Since for all $t\geq0$,
$$\begin{array}{ll}
g(t)&=t^2\left(t^2b\lambda
\left(\ds\int_{\R^3}(|Du|^2+V(x)|u|^2)\right)^2-t\ds\int_{\R^3}|u|^3+(a-1)\ds\int_{\R^3}
(|Du|^2+V(x)|u|^2)\right)\\[5mm]
&\geq0~~~~\hbox{if}~~~\left(\ds\int_{\R^3}|u|^3\right)^2\leq4(a-1)b\lambda\left(\ds\int_{\R^3}(|Du|^2+V(x)|u|^2)\right)^3,
\end{array}
$$
then by \eqref{5.1}, there exists $\lambda_0=\frac{1}{4b(a-1)C^3}>0$
such that  for all $\lambda\geq\lambda_0$ and $t\geq0$, $g(t)\geq0$,
then $g(1)\geq0$, i.e.
$$b\lambda
\left(\ds\int_{\R^3}(|Du|^2+V(x)|u|^2)\right)^2\geq
\ds\int_{\R^3}|u|^3-(a-1)\ds\int_{\R^3}(|Du|^2+V(x)|u|^2).$$
 Hence
$$
\begin{array}{ll}
0&=a\ds\int_{\R^3}
(|Du|^2+V(x)|u|^2)+b\lambda\left(\ds\int_{\R^3}(|Du|^2+V(x)|u|^2\right)^2-\ds\int_{\R^3}|u|^{p+1}\\[5mm]
&\geq a\ds\int_{\R^3}
(|Du|^2+V(x)|u|^2)-(a-1)\ds\int_{\R^3}(|Du|^2+V(x)u^2)+\ds\int_{\R^3}|u|^3-\ds\int_{\R^3}|u|^{p+1}\\[5mm]
&\geq \ds\int_{\R^3} |u|^2+|u|^3-|u|^{p+1}.
\end{array}
$$
Since $1<p\leq2$, then the function $h(t)\triangleq t^2+t^3-t^{p+1}$
is nonnegative for all $t\geq0$ and vanishes only if $t=0$. Hence
$u\equiv0$. The proof is completed.
\end{proof}



 \end{document}